\documentclass[a4paper,11pt]{article}

\usepackage{amsmath}
\usepackage{amssymb}
\usepackage{tabularx}
\usepackage{enumerate}
\usepackage{setspace}
\usepackage{graphicx}
\usepackage{color}
\usepackage{graphicx}
\usepackage{caption}
\usepackage{subfigure}
\usepackage{bm}
\usepackage{float}

\usepackage[pagewise]{lineno}
\newtheorem{theorem}{Theorem}[section]

\newtheorem{lemma}{Lemma}[section]
\newtheorem{remark}{Remark}[section]

\newtheorem{definition}{Definition}[section]
\newenvironment {proof} {\noindent {\bf Proof.}}{\quad $\blacksquare$\par\vspace{3mm}}

\newcommand{\be}{\begin{equation}}
\newcommand{\ee}{\end{equation}}
\newcommand\bes{\begin{eqnarray}}
\newcommand\ees{\end{eqnarray}}
\newcommand{\bess}{\begin{eqnarray*}}
\newcommand{\eess}{\end{eqnarray*}}

\begin{document}
\setlength{\baselineskip}{15.2pt} \pagestyle{myheadings}

\title{ \bf  \LARGE Global solvability in a higher-dimensional chemotaxis system for Alopecia Areata: Nonlinear proliferation versus logistic degradation}
\date{\empty}
\author{
\sl{Haotian Tang$^a$ and Jiashan Zheng$^b$\thanks{Corresponding author.}}\\
{ \normalsize $^a$ School of Mathematics, Renmin University of China,}\\
{ \normalsize Beijing 100872, P. R. China}\\
{ \normalsize $^b$ School of Mathematical and Informational Sciences, Yantai University,}\\
{ \normalsize Yantai, 264005, Shandong, P. R. China}\\
{}\\
}
\maketitle
\renewcommand{\thefootnote}{}
\footnotetext{Email address: haotiantang2022@163.com (H. Tang), zhengjiashan2008@163.com (J. Zheng).}

\begin{quote}
\noindent {\bf Abstract.} {This paper is concerned with the Neumann initial-boundary value problem for the chemotaxis system:
\begin{equation*}
\left\{\aligned
  &  u_{t}=\Delta u-\chi_{1}\nabla\cdot(u\nabla w)+w-\mu_{1}u^{r_{1}},&&x\in\Omega,t>0,\\
  &  v_{t}=\Delta v-\chi_{2}\nabla\cdot(v\nabla w)+w+ruv-\mu_{2}v^{r_{2}},&&x\in\Omega,t>0,\\
  &  w_{t}=\Delta w+u+v-w,&&x\in\Omega,t>0,\\
\endaligned\right.
\end{equation*}
which was initially proposed by Dobreva \emph{et al.} (\cite{DP2020}) to describe the dynamics of hair loss in Alopecia Areata form. Here, $\Omega\subset\mathbb R^{N}$ $(N\geq3)$ is a smooth bounded domain, and the parameters fulfill $\chi_{i}>0$, $\mu_{i}>0$, $r_{i}\geq2$ $(i=1,2)$ and $r>0$. The inherent presence of two positive chemotaxis terms, along with the zero-order nonlinear production term $ruv$, significantly complicates the energy estimation. It is proved that if $r_{1}=r_{2}=2$ and $\min\{\mu_{1},\mu_{2}\}>\mu^{\star}$ or $r_{i}>2$ $(i=1,2)$, this problem admits a global bounded classical solution for all sufficiently smooth initial data. The lower bound is given by $\mu^{\star}=\frac{2(N-2)_{+}}{N}C_{\frac{N}{2}+1}^{\frac{1}{\frac{N}{2}+1}}\max\{\chi_{1},\chi_{2}\}+\left[(\frac{2}{N})^{\frac{2}{N+2}}\frac{N}{N+2}\right]r$, where $C_{\frac{N}{2}+1}$ is a positive constant corresponding to the maximal Sobolev regularity. Furthermore, we demonstrate that the basic assumption $\mu_{i}>0$ $(i=1,2)$ is sufficient to guarantee the global existence of weak solutions for $N\geq3$. Notably, our findings not only extend or refine several existing results (see Remarks 1.1-1.2) but also provide new insights into the weak solution theory of this system for the first time.}


\end{quote}
\newpage
\newcommand\HI{{\bf I}}
\section{Introduction}
\indent
Alopecia areata (AA), commonly known as ``ghost shaving" in China, is an autoimmune disorder disease characterized by localized or complete hair loss. It is not a rare disease, with a global prevalence of about 2\% (\cite{ZL2021}), and a survey in \cite{LG2020} shows that children are more affected than adults. Although AA is not life-threatening like other autoimmune diseases, it imposes a significant psychological burden on patients and negatively affects their social lives (\cite{K2011}). In recent years, the incidence of AA has been rising, but public awareness of it remains insufficient.

A popular hypothesis concerning the etiology of AA (see \cite{GE2012}) is that hair loss results from the immune system's response to autoantigens synthesized in hair follicles (HFs).  Some studies in \cite{GK2006,IS2008} suggest that effector autoreactive CD8$^{+}$ T-cells/NKG2D$^{+}$ cells and effector autoreactive helper CD4$^{+}$ T-cells attack the epithelium of anagen hair follicles, resulting in the shedding of HFs. However, there are no mathematical models that reflect the interaction between HFs and the immune system. In \cite{DP2015}, Dobreva \emph{et al.} constructed an ODE system to describe the population of CD8$^{+}$ T-cells and CD4$^{+}$ T-cells, and then they coupled this system with some equations modelling the hair cycle in \cite{DP2017}. These two models focus on the drivers of $AA$ and are of great significance to explain the mechanisms of AA and formulate therapeutic strategies. Moreover, interferon-gamma (IFN-$\gamma$), which is the most potent inducer of HF immune privilege, induces the chemokine CXCL10 and strongly influences the migration of autoreactive lymphocytes in AA (\cite{IH2013}). Recently, Dobreva \emph{et al.} (\cite{DP2020}) first systematically considered the spatio-temporal patterns of three key components associated with AA progression  and developed a chemotaxis system (fully parabolic):
\begin{equation}\label{model1}
\left\{\aligned
  &  u_{t}=\Delta u-\chi_{1}\nabla\cdot(u\nabla w)+w-\mu_{1}u^{2},&&x\in\Omega,t>0,\\
  &  v_{t}=\Delta v-\chi_{2}\nabla\cdot(v\nabla w)+w+ruv-\mu_{2}v^{2},&&x\in\Omega,t>0,\\
  &  w_{t}=\Delta w+u+v-w,&&x\in\Omega,t>0.\\
\endaligned\right.
\end{equation}
The unknown functions $u=u(x,t)$, $v=v(x,t)$ and $w=w(x,t)$ represent the density of CD4$^{+}$ T-cells, the density of CD8$^{+}$ T-cells and the concentration of IFN-$\gamma$, respectively. The model elaborates on a complex mechanism within the immune microenvironment: IFN-$\gamma$ is produced by CD8$^{+}$ and CD4$^{+}$ T-cells, meanwhile, T-cells are activated in response to IFN-$\gamma$; T-cells tend to migrate towards areas with high concentrations of IFN-$\gamma$, which is a positive chemotaxis effect described by the terms $-\chi_{1}\nabla\cdot(u\nabla w)$ and $-\chi_{2}\nabla\cdot(v\nabla w)$; CD4$^{+}$ T cells act as modifiers to help the proliferation of CD8$^{+}$ T-cells. From a mathematical point of view, two major challenges arise in the analysis of this three-component system: one is that both types of immune cells are activated by IFN-$\gamma$ and undergo chemotaxis, and the other is the presence of the zero-order nonlinear production term $ruv$.

We review some mathematical results for model \eqref{model1}. In \cite{LT2021}, Lou and Tao demonstrated that arbitrarily small $\mu_{i}>0$ can guarantee the global existence and boundedness of classical solution when $N=2$, while when $N=3$, suitably large $\mu_{i}$ $\left(\mu_{i}>16+8\chi_{i}^{2}+\frac{r}{2}, \mu_{1}\mu_{2}^{2}>\frac{4}{27}r^{3}\right)$ is required to prevent blow-up. Moreover, they established the globally asymptotic stability of unique positive equilibrium under certain special parameter conditions. Subsequently, the findings in \cite{SZ2023!,SZ2023!!} further revealed the impact of chemotactic coefficients $\chi_{i}$ on the stability, the instability and the bifurcations. In the case $N\geq4$, the first boundedness result was proved in \cite{ZX2023} for sufficiently large $\mu_{i}$. Later, it was shown in \cite{ZF2024} that sufficiently small $\chi_{i}$ can ensure the boundedness of solutions when $N\geq1$. When the third equation in system \eqref{model1} is replaced by $0=\Delta w+u+v-w$, Tao and Xu (\cite{TX2022}) obtained the global existence result under the conditions $\mu_{1}>\frac{(N-2)_{+}}{N}(2\chi_{1}+\frac{\chi_{2}}{2})+\frac{r}{2}$ and $\mu_{2}>\frac{(N-2)_{+}}{N}(2\chi_{2}+\frac{\chi_{1}}{2})+r$. Then the conditions for $\mu_{i}$ were improved in \cite{SZ2023}. For more studies of system \eqref{model1}, we refer readers to \cite{QZ2024,SZ2023,ZF2025!,ZF2025} for detailed discussions on the impact of complex mechanisms such as nonlinear diffusion, signal-dependent and generalized logistic source on the global existence and boundedness of solutions.

Based on the model \eqref{model1}, Shan and Yang (\cite{SY2025}) recently investigated the global solvability of classical solutions to a quasilinear chemotaxis model incorporating volume-filling effects. Notably, their Theorem 1.2 examines how strong logistic damping can prevent blow-up of solutions in any dimensional domains, but the proof heavily relies on the restrictive conditions $r_{1}\geq r_{2}$ and $r_{i}\leq1+\frac{2(N+2)}{N}$ ($i=1,2$). While the above research results have significantly advanced the mathematical understanding of AA progression, several research gaps still remain unresolved. In this paper, we would like to consider the following questions:\vspace{0.5em}\\
\hspace*{-0.4cm}(Q1) \emph{Can we provide a unified and quantitative description of the relationship between logistic damping rates $\mu_{i}$, chemotaxis coefficients $\chi_{i}$ and proliferation rate $r$ to ensure the boundedness of classical solutions in $\geq$3D?}\vspace{0.5em}\\
\hspace*{-0.4cm}(Q2) \emph{How strong must two independent generalized logistic damping effects be to prevent the blow-up in \eqref{model1} with homogeneous Neumann boundary conditions when $N\geq3$?}\vspace{0.5em}

In \cite{SY2025}, the restrictive conditions $r_{1}\geq r_{2}$ and $r_{i}\leq1+\frac{2(N+2)}{N}$ ($i=1,2$) imply that the two logistic terms are dependent and the upper bounds of $r_{1}$ and $r_{2}$ are constrained by the spatial dimension. In reality, the degradation of the two T-cells populations is independent, and higher values of $r_{1}$ and $r_{2}$ more effectively suppress chemotactic aggression. As we know, the (generalized) logistic source plays an important role in the prevention of blow-ups in various chemotaxis models (see e.g. \cite{LZ2023,QK2022,W2010a,WK2016,X2018!,ZK2019,ZX2024}). Nevertheless, it is observed that few scholars have focused on establishing the global existence and boundedness of classical solutions for model \eqref{model1} and its variants when $N\geq3$. In particular, comprehensive quantitative analysis of logistic damping role within this context appears to be largely absent.\vspace{0.5em}\\
\hspace*{-0.4cm}(Q3) \emph{Can the natural conditions $\mu_{i}>0$ ($i=1,2$) guarantee the global existence of solutions in the weak sense?}\vspace{0.5em}

For the minimal Keller-Segel system, the authors (\cite{L2015,ZL2018}) established the global existence of weak solutions in higher-dimensional $(N\geq3)$ domains for arbitrarily small values of $\mu>0$. Subsequently, global generalized solvability was also demonstrated for a three-dimensional Keller-Segel-Navier-Stokes system with any $\mu>0$ (see \cite{W2019}). For an attraction-repulsion system, the global weak solutions were constructed for any $\mu>0$ as well (see \cite{JJ2024}). However, owing to the increased complexity of the mechanisms in our three-component system \eqref{model1}, the existence of the corresponding solutions in higher dimensions remains unclear when $\mu_{1}$ and $\mu_{2}$ are both sufficiently small.\vspace{0.5em}

Relying on an in-depth understanding of these three questions, we investigate the initial-boundary value problem for the chemotaxis system modelling a pattern of AA dynamics:
\begin{equation}\label{model}
\left\{\aligned
  &  u_{t}=\Delta u-\chi_{1}\nabla\cdot(u\nabla w)+w-\mu_{1}u^{r_{1}},&&x\in\Omega,t>0,\\
  &  v_{t}=\Delta v-\chi_{2}\nabla\cdot(v\nabla w)+w+ruv-\mu_{2}v^{r_{2}},&&x\in\Omega,t>0,\\
  &  w_{t}=\Delta w+u+v-w,&&x\in\Omega,t>0,\\
  &  \frac{\partial u}{\partial \nu}=\frac{\partial v}{\partial \nu}=\frac{\partial w}{\partial \nu}=0,&&x\in\partial\Omega,t>0,\\
  &  u(x,0)=u_{0}(x),v(x,0)=v_{0}(x),w(x,0)=w_{0}(x),&&x\in\Omega,
\endaligned\right.
\end{equation}
where $\Omega\subset\mathbb R^{N}$ $(N\geq3)$ is a bounded domain with smooth boundary $\partial\Omega$. Here, the parameters fulfill $\chi_{i}>0$, $\mu_{i}>0$, $r_{i}\geq2$ $(i=1,2)$ and $r>0$, and the initial data satisfies
\begin{equation}\label{initial}
\left\{\aligned
  &  u_{0}\in C^{0}(\overline{\Omega}) \ \ \mbox{with} \ \ u_{0}\geq, \not\equiv0, \\
  &  v_{0}\in C^{0}(\overline{\Omega}) \ \ \mbox{with} \ \ v_{0}\geq0,\\
  & w_{0}\in W^{1,\infty}(\Omega) \ \ \mbox{with} \ \ w_{0}\geq0.\\
\endaligned\right.
\end{equation}
It is worth noting that the strong coupling among different variables poses significant analytical challenges in handling two chemotactic cross-diffusion terms and the nonlinear term $ruv$ simultaneously. We employ precise energy estimates in conjunction with the maximal Sobolev regularity theory to establish the global existence and boundedness of classical solution. Furthermore, building upon the ideas in \cite{L2015,ZL2018}, we fully explore the intrinsic relationships among the solution components to improve the corresponding weak solution theory for model \eqref{model} through the application of the Aubin-Lions lemma.

We first show the pivotal role of the (generalized) logistic damping in ensuring the global existence and boundedness of classical solutions when $N\geq3$.
\begin{theorem}\label{theorem1}
Let $\Omega\subset\mathbb R^{N}$ $(N\geq3)$ be a bounded domain with smooth boundary, and suppose that the nonnegative initial data $(u_{0},v_{0},w_{0})$ fulfills \eqref{initial}. If either $r_{1}=r_{2}=2$ and $\min\{\mu_{1},\mu_{2}\}>\frac{2(N-2)_{+}}{N}C_{\frac{N}{2}+1}^{\frac{1}{\frac{N}{2}+1}}\max\{\chi_{1},\chi_{2}\}+\left[(\frac{2}{N})^{\frac{2}{N+2}}\frac{N}{N+2}\right]r$, or $r_{i}>2$ $(i=1,2)$, there exist uniquely determined functions
\begin{equation}\label{result}
\left\{\aligned
&u\in C^{0}(\overline{\Omega}\times[0,\infty))\cap C^{2,1}(\overline{\Omega}\times(0,\infty)),
&&\\
&v\in C^{0}(\overline{\Omega}\times[0,\infty))\cap C^{2,1}(\overline{\Omega}\times(0,\infty)),
&&\\
&w\in C^{0}(\overline{\Omega}\times[0,\infty))\cap C^{2,1}(\overline{\Omega}\times(0,\infty))
\endaligned\right.
\end{equation}
such that the triple $(u,v,w)$ forms a classical solution to \eqref{model}.
Moreover, $(u, v, w)$ is bounded in $\Omega\times(0,\infty)$ in the sense that there exists a constant $C>0$ satisfying
\begin{equation*}
\|u(\cdot,t)\|_{L^{\infty}(\Omega)}+\|v(\cdot,t)\|_{L^{\infty}(\Omega)}
+\|w(\cdot,t)\|_{W^{1,\infty}(\Omega)}\leq C\ \ \ \mbox{for all}\ t>0.
\end{equation*}
\end{theorem}
\begin{remark}
(Notes on the global existence of \textbf{classical solution})\\
(1) When $r_{1}=r_{2}=2$, we give a unified lower bound of two logistic damping rates $\mu^{\star}(\chi_{1},\chi_{2},r,N)$ to ensure the global existence and boundedness of classical solution in arbitrary higher-dimensional $(N\geq3)$ nonconvex domains. This global existence result extends the previous 3-D result in \cite{LT2021} and 4/5-D result in \cite{ZX2023}.\\
(2) The lower bound of $\mu_{1}$ or $\mu_{2}$ may not depend on $r$ by a small modification in the application of Young's inequality. Obviously, our result improves the results of \cite{LT2021} and \cite{ZX2023}, which the assumption $\mu_{1}\mu_{2}^{2}>\frac{4}{27}r^{3}$ is intrinsically required and the lower bounds of $\mu_{1}$ and $\mu_{2}$ depend on $\varepsilon$, respectively.\\
(3) Our result improves upon Theorem 1.2 of \cite{SY2025} in two key aspects: we provide a more precise description of $\mu^{\star}$, and we remove the strict constraints $r_{1}\geq r_{2}$ and $r_{i}\leq1+\frac{2(N+2)}{N}$ ($i=1,2$) required in \cite{SY2025}.\\
(4) In \cite{ZZ2024}, the global existence of classical solution to \eqref{model} has been proved under the conditions that  $r_{i}>\max\left\{2,\frac{N}{2}+1\right\}$ or $r_{i}>2+\frac{N-2}{N+2}$ $(i=1,2)$, while these assumptions are relaxed to $r_{i}>2$ in Theorem \ref{theorem1}.
\end{remark}

The following result reveals the global existence of weak solution for arbitrary $\mu_{i}>0$. To begin with, we need to introduce the concept of weak solution.
\begin{definition}\label{weak solution}
Let $\Omega\subset\mathbb R^{N}$ $(N\geq3)$ with smooth boundary $\partial\Omega$, and suppose that the initial data satisfies \eqref{initial}. Then a triple $(u,v,w)$ of nonnegative functions
\begin{equation}\label{weak}
\left\{\aligned
&u\in L^{2}_{loc}([0,\infty);L^{2}(\Omega)),
&&\\
&v\in L^{2}_{loc}([0,\infty);L^{2}(\Omega)),
&&\\
&w\in L^{1}_{loc}([0,\infty);W^{1,1}(\Omega))
\endaligned\right.
\end{equation}
will be called a global weak solution of \eqref{model} if
\begin{equation}
\nabla u, \ \ \nabla v, \ \ u \nabla w \ \ \mbox{and}\ \ v \nabla w \ \ \mbox{belong to} \ \ L^{1}_{loc}(\overline{\Omega}\times[0,\infty);\mathbb{R}^{N}),
\end{equation}
and if the identities
\begin{equation}\label{i1}
-\int_{0}^{T}\int_{\Omega}u\varphi_{t}-\int_{\Omega}u_{0}\varphi(\cdot, 0)=-\int_{0}^{T}\int_{\Omega}\nabla u\cdot\nabla\varphi+\chi_{1}\int_{0}^{T}\int_{\Omega}u\nabla w\cdot\nabla\varphi+\int_{0}^{T}\int_{\Omega}(w-\mu_{1}u^{2})\varphi,
\end{equation}
\begin{equation}\label{i2}
-\int_{0}^{T}\int_{\Omega}v\varphi_{t}-\int_{\Omega}v_{0}\varphi(\cdot, 0)=-\int_{0}^{T}\int_{\Omega}\nabla v\cdot\nabla\varphi+\chi_{2}\int_{0}^{T}\int_{\Omega}v\nabla w\cdot\nabla\varphi+\int_{0}^{T}\int_{\Omega}(w+ruv-\mu_{2}v^{2})\varphi
\end{equation}
as well as
\begin{equation}\label{i3}
-\int_{0}^{T}\int_{\Omega}w\varphi_{t}-\int_{\Omega}w_{0}\varphi(\cdot, 0)=-\int_{0}^{T}\int_{\Omega}\nabla w\cdot\nabla\varphi+\int_{0}^{T}\int_{\Omega}(u+v-w)\varphi
\end{equation}
hold for each $\varphi\in C_{0}^{\infty}(\overline{\Omega}\times[0,\infty))$.
\end{definition}
\begin{theorem}\label{theorem2}
Let $\Omega\subset\mathbb R^{N}$ $(N\geq3)$ be a bounded domain with smooth boundary and let the initial data satisfies \eqref{initial}. Suppose that $r_{1}=r_{2}=2$. Then for any $\mu_{i}>0$ $(i=1,2)$, the system \eqref{model} admits at least one global weak solution $(u,v,w)$ in the sense of Definition \ref{weak solution}.
\end{theorem}

\begin{remark} (Notes on the global existence of \textbf{weak solution})\\
(1) To the best of our knowledge, Theorem \ref{theorem2} is the first result concerning the weak solution theory for system \eqref{model}.\\
(2) Although our three-component system is more complex than the minimal KS system, we still establish the global existence of weak solution in higher dimensions $(N\geq3)$ for any small values of $\mu_{i}>0$ $(i=1,2)$. Specifically, the proof does not require the restriction on the convexity of $\Omega$.\\
(3) The eventual smoothness of global weak solution will be studied in our future work.
\end{remark}

The remainder of this paper is structured as follows. In Section 2, we review some preliminary results that are essential to our subsequent analysis. In Section 3, we shall show the global solvability of the model \eqref{model} under two distinct logistic damping roles via testing procedure and the maximal Sobolev regularity argument. In the case of $r_{1}=r_{2}=2$ and $\min\{\mu_{1},\mu_{2}\}>\mu^{\ast}$, we raise the \emph{a priori} estimates of solutions from $L^{1}(\Omega)\rightarrow L^{\frac{N}{2}}(\Omega)\rightarrow L^{\frac{N}{2}+\varepsilon}(\Omega)\rightarrow L^{p}(\Omega)$ (for any $p>1$) by means of a crucial auxiliary Lemma \ref{a}. (see Lemmas \ref{lemmaLp0} and \ref{lemmaLp}). In the case of $r_{i}>2$ ($i=1,2$), we develop $L^{p}$-estimate of $u$ and $L^{q}$-estimate of $v$ by leveraging the characteristic of generalized logistic source (see Lemma \ref{p-q}). These estimates enable us to derive the boundedness results according to a Moser-type iterative argument. Section 4 is devoted to proving the global existence of weak solution in the sense of Definition \ref{weak solution}. First, we employ the standard $L^{p}$ testing procedure to establish the global solvability of the regularized problem \eqref{weak model} in the classical sense (see Lemma \ref{W3}). Based on some $\varepsilon$-independent estimates, we further obtain essential spatial-temporal estimates and derive several regularity results for time derivatives (see Lemmas \ref{W2} and \ref{W4}), which allow us to prove specific compactness properties via an Aubin-Lions type lemma. Finally, we complete the proof of Theorem \ref{theorem2} through an appropriate limit procedure.\\
\textbf{Notations.} Throughout this paper, various positive constants are denoted by $C$, $C_{\ast}$, $C_{\ast\ast}$ or $C_{i}$ $(i=1, 2, \cdot\cdot\cdot)$. Moreover, we omit the spatial integration symbol $dx$ for brevity.

\section{Preliminaries}
\setcounter{equation}{0}
\indent

This section contains several lemmas that play an important role in our \emph{a priori} estimates. We start with the widely applied Gagliardo-Nirenberg interpolation inequality.

\begin{lemma}\label{GN}
(Gagliardo-Nirenberg inequality \cite{N1966, XZ2019}) Let $\Omega\subset\mathbb R^{N}$ be a bounded domain with smooth boundary. Suppose $p\geq1$ and $q\in(0,p]$. Then there exists a positive constant $C_{GN}=C(p,q,N,\Omega)$ such that
\begin{center}
$\|w\|_{L^{p}(\Omega)}\leq C_{GN}(\|\nabla w\|_{L^{2}(\Omega)}^{\alpha}\|w\|_{L^{q}(\Omega)}^{1-\alpha}+\|w\|_{L^{q}(\Omega)})$.
\end{center}
for any functions $w\in H^{1}(\Omega)\cap L^{q}(\Omega)$, where $\alpha$ is given by
\begin{center}
$\alpha=\frac{\frac{1}{p}-\frac{1}{q}}{\frac{1}{2}-\frac{1}{N}-\frac{1}{q}}\in(0,1)$.
\end{center}
\end{lemma}

The selection of certain $\varepsilon$ in Young's inequality within Lemma \ref{lemmaLp0} relies on calculating the minimum value of the following function.
\begin{lemma}\label{a}
Let
\begin{equation*}
A_{1}=\frac{1}{\delta+1}\left(\frac{\delta+1}{\delta}\right)^{-\delta}\left(\frac{\delta-1}{\delta}\right)^{\delta+1}
\end{equation*}
with any fixed $\delta>1$. Suppose
\begin{equation*}
H(y)=y+A_{1}y^{-\delta}(2\chi)^{\delta+1}C_{\delta+1}
\end{equation*}
for $y>0$, where some fixed constants $\chi>0$ and $C_{\delta+1}>0$. Then one has
\begin{equation}
\min \limits_{y>0}H(y)=\frac{2(\delta-1)}{\delta}(C_{\delta+1})^{\frac{1}{\delta+1}}\chi.
\end{equation}
\end{lemma}
\begin{proof}
We calculate that
\begin{equation*}
H'(y)=1-A_{1}\delta C_{\delta+1}\left(\frac{2\chi}{y}\right)^{\delta+1}.
\end{equation*}
Then letting $H'(y)=0$, we have
\begin{equation}\label{12345}
y=2(A_{1}\delta C_{\delta+1})^{\frac{1}{\delta+1}}\chi.
\end{equation}
Since $\lim\limits_{y\rightarrow0^{+}}H(y)=+\infty$ and $\lim\limits_{y\rightarrow+\infty}H(y)=+\infty$, from \eqref{12345} we derive that
\begin{equation*}
\begin{aligned}
\min \limits_{y>0}H(y)=H[2(A_{1}\delta C_{\delta+1})^{\frac{1}{\delta+1}}\chi]=&2(A_{1}C_{\delta+1})^{\frac{1}{\delta+1}}(\delta^{\frac{1}{\delta+1}}+\delta^{-\frac{\delta}{\delta+1}})\chi\\
&=\frac{2(\delta-1)}{\delta}(C_{\delta+1})^{\frac{1}{\delta+1}}\chi.
\end{aligned}
\end{equation*}
\end{proof}

We then show a boundedness property for solutions to an auxiliary differential inequality.
\begin{lemma}\label{c} (\cite{ZK2021}) Let $T>0$, $\tau\in(0,T)$, $\alpha>0$ and $B>0$. Suppose that $z: [0,T)\rightarrow[0,\infty)$ is absolutely continuous and satisfies
\begin{equation*}
z'(t)+Az^{\alpha}(t)\leq h(t)\ \ \ \mbox{for a.e.}\ t\in(0,T)
\end{equation*}
with some nonnegative function $h\in L_{loc}^{1}([0,T))$. If
\begin{equation*}
\int_{t}^{t+\tau}h(s)ds\leq B\ \ \ \mbox{for all}\ t\in(0,T-\tau),
\end{equation*}
then one can find a positive constant $C=\max \left\{z_{0}+B,\frac{1}{\tau^{\frac{1}{\alpha}}}(\frac{B}{A})^{\frac{1}{\alpha}}+2B \right\}$ such that
\begin{equation*}
z(t)\leq C\ \ \ \mbox{for all}\ t\in(0,T).
\end{equation*}
\end{lemma}

The following statement about the maximal Sobolev regularity theory is a powerful tool for estimating $\int_{\Omega}|\Delta w|^{p}+\int_{\Omega}w^{p}$.
\begin{lemma}\label{b}
(cf. \cite{KZ2018}, \cite{ZL2018}) Let $\gamma\in(1,+\infty)$ and $g\in L^{\gamma}((0,T); L^{\gamma}(\Omega))$. Consider the following initial boundary problem:
\begin{equation*}
\left\{\aligned
&v_{t}-\Delta v+v=g,&&(x,t)\in\Omega\times(0,T),
&&\\
&\frac{\partial v}{\partial \nu}=0,&&(x,t)\in\partial\Omega\times(0,T),
&&\\
&v(x,0)=v_{0}(x),&&x\in\Omega.
\endaligned\right.
\end{equation*}
For each $v_{0}\in W^{2,\gamma}(\Omega)$ with $\frac{\partial v_{0}}{\partial\nu}=0$, there exists a unique solution $v\in W^{1,\gamma}((0,T);L^{\gamma}(\Omega))\cap L^{\gamma}((0,T);W^{2,\gamma}(\Omega))$. Moreover, if $s_{0}\in[0,T)$ and $v_{0}\in W^{2,\gamma}(\Omega)$ with $\frac{\partial v_{0}}{\partial\nu}=0$, then there exists a positive constant $C_{\gamma}$, which depends continuously on $\gamma$, such that
\begin{equation}\label{b1}
\int_{s_{0}}^{T}e^{\gamma s}\|v(\cdot,s)\|^{\gamma}_{W^{2,\gamma}(\Omega)}ds
\leq C_{\gamma}\left(\int_{s_{0}}^{T}e^{\gamma s}\|g(\cdot,s)\|^{\gamma}_{L^{\gamma}(\Omega)}ds+e^{\gamma s_{0}}\|v(\cdot,s_{0})\|^{\gamma}_{W^{2,\gamma}(\Omega)}\right).
\end{equation}
\end{lemma}
\begin{proof}
It remains to establish the continuous dependence of $C_\gamma$ on $\gamma$. Let $u(x,t)=e^t v(x,t)$ and $f(x,t)=e^t g(x,t)$. Then $u$ satisfies
\begin{equation}\label{pp}
\left\{\aligned
&u_t-\Delta u=f,&&(x,t)\in \Omega\times(s_0,T),
&&\\
&\frac{\partial u}{\partial \nu}=0,&&(x,t)\in\partial\Omega\times(s_0,T),
&&\\
&u(x,s_0)=a(x),&&x\in\Omega.
\endaligned\right.
\end{equation}
For fixed $s_0$ and $T$, we define
\begin{equation*}
\mathcal S_\gamma:(f,a)\mapsto u
\end{equation*}
as the solution operator associated with \eqref{pp}. More precisely, \(\mathcal S_\gamma\) maps
\begin{equation*}
X_\gamma:=L^\gamma((s_0,T);L^\gamma(\Omega))\times W^{2,\gamma}_\nu(\Omega)
\end{equation*}
into
\begin{equation*}
Y_\gamma:=L^\gamma((s_0,T);W^{2,\gamma}(\Omega)),
\end{equation*}
where
\begin{equation*}
W^{2,\gamma}_\nu(\Omega)
:=\left\{
\phi\in W^{2,\gamma}(\Omega):\frac{\partial \phi}{\partial\nu}=0\ \text{on } \partial\Omega\right\}.
\end{equation*}
By the parabolic maximal regularity estimate, one has
\begin{equation}\label{ppp}
\|\mathcal S_\gamma(f,a)\|_{Y_\gamma}
\leq K_\gamma
\left(\|f\|_{L^\gamma((s_0,T);L^\gamma(\Omega))}+\|a\|_{W^{2,\gamma}(\Omega)}\right)\ \ \ \mbox{for all}\ \gamma\in(1,\infty)
\end{equation}
with $K_\gamma>0$. Let $1<\gamma_1<\gamma_2<\infty$, $0<\theta<1$, and define
$\gamma$ by
\begin{equation*}
\frac1\gamma=
\frac{1-\theta}{\gamma_1}+\frac{\theta}{\gamma_2}.
\end{equation*}
The complex interpolation identities
\begin{equation*}
[L^{\gamma_1}(\Omega),L^{\gamma_2}(\Omega)]_\theta
=L^\gamma(\Omega), \ \ [W^{2,\gamma_1}(\Omega),W^{2,\gamma_2}(\Omega)]_\theta
=W^{2,\gamma}(\Omega)
\end{equation*}
together with the corresponding interpolation property of the Neumann subspaces $W^{2,\gamma}_\nu(\Omega)$ provide a constant $M(\gamma_1,\gamma_2,\theta)>0$ fulfilling
\begin{equation}\label{pppp}
\|\mathcal S_\gamma\|_{X_\gamma\to Y_\gamma}
\leq
M(\gamma_1,\gamma_2,\theta)
\,
\|\mathcal S_{\gamma_1}\|_{X_{\gamma_1}\to Y_{\gamma_1}}^{1-\theta}
\,
\|\mathcal S_{\gamma_2}\|_{X_{\gamma_2}\to Y_{\gamma_2}}^\theta.
\end{equation}
Based on \eqref{ppp}, we deduce from \eqref{pppp} that on every compact
subinterval of $(1,\infty)$, $\|\mathcal S_\gamma\|_{X_\gamma\to Y_\gamma}$ are locally bounded. Consequently, after enlarging the constants if necessary, one may choose a continuous function $\widetilde K_\gamma$ such that
\begin{equation}\label{ppppp}
\|\mathcal S_\gamma\|_{X_\gamma\to Y_\gamma}
\leq\widetilde K_\gamma \ \ \ \mbox{for all}\ \gamma\in(1,\infty).
\end{equation}
According to the elementary inequality $(A+B)^\gamma\leq 2^{\gamma-1}(A^\gamma+B^\gamma)$ and the definitions of $u(x,t)$, $f(x,t)$ and $a(x)$, the combination of \eqref{ppp} and \eqref{ppppp} yields that
\begin{equation*}
\begin{aligned}
\int_{s_0}^T e^{\gamma s}
\|v(\cdot,s)\|_{W^{2,\gamma}(\Omega)}^\gamma\,ds
&\leq2^{\gamma-1}\widetilde K_\gamma^\gamma
\left(\int_{s_0}^T e^{\gamma s}
\|g(\cdot,s)\|_{L^\gamma(\Omega)}^\gamma\,ds+e^{\gamma s_0}\|v(\cdot,s_0)\|_{W^{2,\gamma}(\Omega)}^\gamma\right).
\end{aligned}
\end{equation*}
Thus \eqref{b1} holds with $C_\gamma:=2^{\gamma-1}\widetilde K_\gamma^\gamma$. Since \(\widetilde K_\gamma\) is chosen to be continuous in \(\gamma\), it follows that $\gamma\mapsto C_\gamma$ is continuous on $(1,\infty)$. This proves the additional assertion.
\end{proof}

In order to improve the regularity of $w$, we present the following useful estimates.
\begin{lemma}\label{d}
(\cite{HW2005},\cite{KS2008}) Let $\gamma\in(1,+\infty)$ and $g\in L^{\infty}((0,T_{max});L^{\gamma}(\Omega))$. Suppose that $v$ is a solution of the initial boundary problem
\begin{equation*}
\left\{\aligned
&v_{t}-\Delta v+v=g,
&&\\
&\frac{\partial v}{\partial \nu}=0,
&&\\
&v(x,0)=v_{0}(x).
\endaligned\right.
\end{equation*}
Then there exists a positive constant $C$ independent of $t$ such that
\begin{equation*}
\|v(\cdot,t)\|_{W^{1,q}(\Omega)}\leq C\ \ \ \mbox{for all}\ t \in(0,T_{max}),
\end{equation*}
where
\begin{equation*}
q\in\begin{cases}[1,\frac{N\gamma}{N-\gamma})&\mbox{if} \ \ \gamma\leq N,\\
[1,\infty]& \mbox{if} \ \ \gamma> N. \end{cases}
\end{equation*}
\end{lemma}

\section{Global existence of the classical solution}
\setcounter{equation}{0}
\indent

In this section, we shall prove the system \eqref{model} possesses a global classical solution which is bounded in two cases: (\romannumeral 1) when $r_{1}=r_{2}=2$ with sufficiently large $\mu_{1}$ and $\mu_{2}$, and (\romannumeral 2) when $r_{i}>2$ ($i=1,2$) without restrictions on $\mu_{1}$ and $\mu_{2}$. Our analysis begins with the local existence result for classical solutions, which was previously established in \cite{LT2021}.
\begin{lemma}\label{Local1}
Let $\Omega\subset\mathbb R^{N}$($N\geq3$) be a bounded domain with smooth boundary. For any given initial data $(u_{0},v_{0},w_{0})$ fulfilling \eqref{initial}, there exist a maximal
existence time $T_{max}\in(0,\infty]$ and a unique triple $(u,v,w)$ of nonnegative functions
\begin{equation}\label{local1}
\left\{\aligned
&u\in C^{0}(\overline{\Omega}\times[0,T_{max}))\cap C^{2,1}(\overline{\Omega}\times(0,T_{max})),
&&\\
&v\in C^{0}(\overline{\Omega}\times[0,T_{max}))\cap C^{2,1}(\overline{\Omega}\times(0,T_{max})),
&&\\
&w\in C^{0}(\overline{\Omega}\times[0,T_{max}))\cap C^{2,1}(\overline{\Omega}\times(0,T_{max})),
\endaligned\right.
\end{equation}
which solves \eqref{model} in the classical sense in $\Omega\times(0,T_{max})$. Furthermore, if $T_{max}<\infty$, then one has
\begin{equation}\label{local2}
\|u(\cdot,t)\|_{L^{\infty}(\Omega)}+\|v(\cdot,t)\|_{L^{\infty}(\Omega)}\rightarrow\infty\ \ \mbox{as}\ \  t\nearrow T_{max}.
\end{equation}
\end{lemma}

In view of \eqref{local1}, there exists a positive constant $K$ such that for any $s_{0}\in(0,T_{max})$ with $s_{0}\leq 1$,
\begin{equation}\label{s0}
\|u(\cdot, \tau)\|_{L^{\infty}(\Omega)}\leq K, \ \ \ \|v(\cdot, \tau)\|_{L^{\infty}(\Omega)}\leq K\ \ \ \mbox{for all}\ \tau\in[0,s_{0}].
\end{equation}

As a starting point for \emph{a priori} estimates, the basic $L^{1}$-property can be derived as below. Although a rigorous proof is available in Lemma 2.2 of \cite{LT2021}, we provide our own version for completeness and to facilitate the subsequent analysis of Lemma \ref{W1}.
\begin{lemma}\label{lemmaL1}
Suppose $r_{i}\geq2$ $(i=1,2)$. Then there exists a constant $C>0$ such that the solution to model \eqref{model} satisfies
\begin{equation}\label{!}
\|u(\cdot, t)\|_{L^{1}(\Omega)}+\|v(\cdot, t)\|_{L^{1}(\Omega)}+\|w(\cdot, t)\|_{L^{1}(\Omega)}\leq C\ \ \ \mbox{for all}\ t \in(0,T_{max}).
\end{equation}
\end{lemma}
\begin{proof}
Under the Neumann boundary conditions, some integrations by parts show that
\begin{equation}\label{43}
\frac{d}{dt}\int_{\Omega}u+\int_{\Omega}u=\int_{\Omega}w+\int_{\Omega}u-\mu_{1}\int_{\Omega}u^{r_{1}} \ \ \ \mbox{for all}\ t \in(0,T_{max}),
\end{equation}
\begin{equation}\label{44}
\frac{d}{dt}\int_{\Omega}v+\int_{\Omega}v=\int_{\Omega}w+\int_{\Omega}v+r\int_{\Omega}uv-\mu_{2}\int_{\Omega}v^{r_{2}} \ \ \ \mbox{for all}\ t \in(0,T_{max})
\end{equation}
and
\begin{equation}\label{45}
\frac{d}{dt}\int_{\Omega}w+\int_{\Omega}w=\int_{\Omega}u+\int_{\Omega}v\ \ \ \mbox{for all}\ t \in(0,T_{max}).
\end{equation}
For the term $r\int_{\Omega}uv$, we apply Young's inequality to derive
\begin{equation}\label{46}
r\int_{\Omega}uv\leq \frac{\mu_{2}}{2}\int_{\Omega}v^{r_{2}}+L\int_{\Omega}u^{\frac{r_{2}}{r_{2}-1}}\ \ \ \mbox{for all}\ t \in(0,T_{max}),
\end{equation}
where $L=\frac{r_{2}-1}{r_{2}}\left(\frac{\mu_{2}r_{2}}{2}\right)^{-\frac{1}{r_{2}-1}}r^{\frac{r_{2}}{r_{2}-1}}>0.$ In conjunction with \eqref{43}-\eqref{46}, for all $t \in(0,T_{max})$ a straightforward computation yields
\begin{equation}\label{47}
\begin{aligned}
&\frac{d}{dt}\left(\frac{2L}{\mu_{1}}\int_{\Omega}u+\int_{\Omega}v+\frac{4L+2\mu_{1}}{\mu_{1}}\int_{\Omega}w\right)+\frac{2L}{\mu_{1}}\int_{\Omega}u+\int_{\Omega}v+\frac{2L+\mu_{1}}{\mu_{1}}\int_{\Omega}w\\
\leq&\frac{6L+2\mu_{1}}{\mu_{1}}\int_{\Omega}u+\frac{4L+3\mu_{1}}{\mu_{1}}\int_{\Omega}v+L\int_{\Omega}u^{\frac{r_{2}}{r_{2}-1}}-2L\int_{\Omega}u^{r_{1}}-\frac{\mu_{2}}{2}\int_{\Omega}v^{r_{2}}.
\end{aligned}
\end{equation}
Since $r_{i}\geq2$ $(i=1,2)$ implies $r_{1}\geq\frac{r_{2}}{r_{2}-1}$, we can find some positive constants $C_{i}$ $(i=1, ..., 3)$ fulfilling
\begin{equation*}
L\int_{\Omega}u^{\frac{r_{2}}{r_{2}-1}}\leq L\int_{\Omega}u^{r_{1}}+C_{1}\ \ \ \mbox{for all}\ t \in(0,T_{max})
\end{equation*}
and
\begin{equation*}
\frac{6L+2\mu_{1}}{\mu_{1}}\int_{\Omega}u\leq \frac{L}{2}\int_{\Omega}u^{r_{1}}+C_{2}\ \ \ \mbox{for all}\ t \in(0,T_{max})
\end{equation*}
as well as
\begin{equation*}
\frac{4L+3\mu_{1}}{\mu_{1}}\int_{\Omega}v\leq \frac{\mu_{2}}{4}\int_{\Omega}v^{r_{2}}+C_{3}\ \ \ \mbox{for all}\ t \in(0,T_{max})
\end{equation*}
thanks to Young's inequality, which update \eqref{47} as
\begin{equation}\label{51}
\begin{aligned}
&\frac{d}{dt}\left(\frac{2L}{\mu_{1}}\int_{\Omega}u+\int_{\Omega}v+\frac{4L+2\mu_{1}}{\mu_{1}}\int_{\Omega}w\right)+\frac{2L}{\mu_{1}}\int_{\Omega}u+\int_{\Omega}v+\frac{2L+\mu_{1}}{\mu_{1}}\int_{\Omega}w\\
\leq &-\frac{L}{2}\int_{\Omega}u^{r_{1}}-\frac{\mu_{2}}{4}\int_{\Omega}v^{r_{2}}+C_{4}\ \ \ \mbox{for all}\ t \in(0,T_{max}),
\end{aligned}
\end{equation}
where some constant $C_{4}=C_{1}+C_{2}+C_{3}>0$. In consequence, this implies that
\begin{equation*}
y(t):=\frac{2L}{\mu_{1}}\int_{\Omega}u(\cdot,t)+\int_{\Omega}v(\cdot,t)+\frac{4L+2\mu_{1}}{\mu_{1}}\int_{\Omega}w(\cdot,t)
\end{equation*}
satisfies
\begin{equation*}
y'(t)+\frac{1}{2}y(t)+\frac{L}{2}\int_{\Omega}u^{r_{1}}+\frac{\mu_{2}}{4}\int_{\Omega}v^{r_{2}}\leq C_{4}\ \ \ \mbox{for all}\ t \in(0,T_{max}),
\end{equation*}
and thus establishes \eqref{!} according to an ODE comparison argument.
\end{proof}

Next, we aim to obtain the higher-order regularity of solutions under two distinct conditions. Due to the structural differences in the logistic source terms, we employ separate bootstrap iteration procedures to raise the regularity of $u$ and $v$ in the following three lemmas.
\begin{lemma}\label{lemmaLp0}
Let $N\geq3$ and $r_{1}=r_{2}=2$. Then for some $q_{0}>\frac{N}{2}$ there exists a constant $C>0$ such that if $\min\{\mu_{1},\mu_{2}\}>\frac{2(N-2)_{+}}{N}C_{\frac{N}{2}+1}^{\frac{1}{\frac{N}{2}+1}}\max\{\chi_{1},\chi_{2}\}+\left[(\frac{2}{N})^{\frac{2}{N+2}}\frac{N}{N+2}\right]r$, we have
\begin{equation}\label{Lp0}
\|u(\cdot, t)\|_{L^{q_{0}}(\Omega)}+\|v(\cdot, t)\|_{L^{q_{0}}(\Omega)}\leq C  \ \ \ \mbox{for all}\ t \in(0,T_{max}).
\end{equation}
\end{lemma}
\begin{proof}
Define $\chi=\max\{\chi_{1},\chi_{2}\}$. Multiplying the first equation by $u^{q-1}$ $(q>1)$ and integrating by parts, one has
\begin{equation}\label{1}
\begin{aligned}
&\frac{1}{q}\frac{d}{dt}\int_{\Omega}u^{q}+\frac{q+1}{q}\int_{\Omega}u^{q}+(q-1)\int_{\Omega}u^{q-2}|\nabla u|^{2}\\
=&-\chi_{1}\int_{\Omega}\nabla\cdot(u\nabla w)u^{q-1}+\int_{\Omega}u^{q-1}w+\frac{q+1}{q}\int_{\Omega}u^{q}-\mu_{1}\int_{\Omega}u^{q+1}\\
\leq&\frac{q-1}{q}\chi\int_{\Omega}u^{q}|\Delta w|+\int_{\Omega}u^{q-1}w+\frac{q+1}{q}\int_{\Omega}u^{q}-\mu_{1}\int_{\Omega}u^{q+1} \ \ \ \mbox{for all}\ t \in(0,T_{max}).
\end{aligned}
\end{equation}
Let
\begin{equation*}
\lambda_{0} :=2(A_{1}C_{q+1}q)^{\frac{1}{q+1}}\chi,
\end{equation*}
where $A_{1}$ is defined as in Lemma \ref{a} ($\delta=q$), and $C_{q+1}$ is given by Lemma \ref{b} ($\gamma=q+1$).
By Young's inequality, for any $\varepsilon_{1}>0$ and $\varepsilon_{2}>0$ we have
\begin{equation}\label{3}
\begin{aligned}
\frac{q-1}{q}\chi\int_{\Omega}u^{q}|\Delta w|&\leq\lambda_{0}\int_{\Omega}u^{q+1}+\frac{1}{q+1}\left[\lambda_{0}\frac{q+1}{q}\right]^{-q}\left[\frac{q-1}{q}\chi\right]^{q+1}\int_{\Omega}|\Delta w|^{q+1}\\
&=\lambda_{0}\int_{\Omega}u^{q+1}+A_{1}\lambda_{0}^{-q}\chi^{q+1}\int_{\Omega}|\Delta w|^{q+1}
\end{aligned}
\end{equation}
and
\begin{equation}\label{4}
\begin{aligned}
&\int_{\Omega}u^{q-1}w+\frac{q+1}{q}\int_{\Omega}u^{q}-\mu_{1}\int_{\Omega}u^{q+1}\\
\leq& C_{1}(\varepsilon_{1},q)\int_{\Omega}w^{\frac{q+1}{2}}+(\varepsilon_{1}+\varepsilon_{2}-\mu_{1})\int_{\Omega}u^{q+1}+C_{2}(\varepsilon_{2},q)\\
\leq& A_{1}\lambda_{0}^{-q}\chi^{q+1}\int_{\Omega}w^{q+1}+(\varepsilon_{1}+\varepsilon_{2}-\mu_{1})\int_{\Omega}u^{q+1}+C_{3}(\varepsilon_{1},\varepsilon_{2},q),
\end{aligned}
\end{equation}
where
\begin{equation*}
C_{1}(\varepsilon_{1},q)=\frac{2}{q+1}\left(\varepsilon_{1}\frac{q+1}{q-1}\right)^{-\frac{q-1}{2}},
\end{equation*}
\begin{equation*}
C_{2}(\varepsilon_{2},q)=\frac{1}{q+1}\left(\varepsilon_{2}\frac{q+1}{q}\right)^{-q}\left(\frac{q+1}{q}\right)^{q+1}|\Omega|
\end{equation*}
and
\begin{equation*}
C_{3}(\varepsilon_{1},\varepsilon_{2},q)=C_{2}(\varepsilon_{2},q)+\frac{1}{2}\left(2A_{1}\lambda_{0}^{-q}\chi^{q+1}\right)^{-1}C_{1}^{2}(\varepsilon_{1},q)|\Omega|,
\end{equation*}
Combining \eqref{3} and \eqref{4} with \eqref{1} yields
\begin{equation}\label{8}
\begin{aligned}
&\frac{1}{q}\frac{d}{dt}\int_{\Omega}u^{q}+\frac{q+1}{q}\int_{\Omega}u^{q}\\
\leq &A_{1}\lambda_{0}^{-q}\chi^{q+1}\left(\int_{\Omega}|\Delta w|^{q+1}+\int_{\Omega}w^{q+1}\right)+(\lambda_{0}+\varepsilon_{1}+\varepsilon_{2}-\mu_{1})\int_{\Omega}u^{q+1}+C_{3}(\varepsilon_{1},\varepsilon_{2},q)
\end{aligned}
\end{equation}
for all $t\in(0,T_{max})$. Multiplying both sides of \eqref{8} by $e^{(q+1)t}$, we obtain
\begin{equation}\label{9}
\begin{aligned}
&\frac{1}{q}\frac{d}{dt}\left(e^{(q+1)t}\|u(\cdot,t)\|_{L^{q}(\Omega)}^{q}\right)\\
\leq &\left[A_{1}\lambda_{0}^{-q}\chi^{q+1}\int_{\Omega}\left(|\Delta w|^{q+1}+w^{q+1}\right)+(\lambda_{0}+\varepsilon_{1}+\varepsilon_{2}-\mu_{1})\int_{\Omega}u^{q+1}+C_{3}(\varepsilon_{1},\varepsilon_{2},q)\right]e^{(q+1)t}.
\end{aligned}
\end{equation}
Integrating \eqref{9} over $[s_{0},t)$, for all $t\in(s_{0},T_{max})$, we see that
\begin{equation}\label{10}
\begin{aligned}
&\frac{1}{q}\|u(\cdot, t)\|_{L^{q}(\Omega)}^{q}\\
\leq &\frac{1}{q}e^{-(q+1)(t-s_{0})}\|u(\cdot, s_{0})\|_{L^{q}(\Omega)}^{q}+A_{1}\lambda_{0}^{-q}\chi^{q+1}\int_{s_{0}}^{t}e^{-(q+1)(t-s)}\int_{\Omega}\left(|\Delta w|^{q+1}+w^{q+1}\right)ds\\
&+(\lambda_{0}+\varepsilon_{1}+\varepsilon_{2}-\mu_{1})\int_{s_{0}}^{t}e^{-(q+1)(t-s)}\int_{\Omega}u^{q+1}ds+C_{3}(\varepsilon_{1},\varepsilon_{2},q)\int_{s_{0}}^{t}e^{-(q+1)(t-s)}ds\\
\leq & A_{1}\lambda_{0}^{-q}\chi^{q+1}\int_{s_{0}}^{t}e^{-(q+1)(t-s)}\int_{\Omega}\left(|\Delta w|^{q+1}+w^{q+1}\right)ds\\
&+(\lambda_{0}+\varepsilon_{1}+\varepsilon_{2}-\mu_{1})\int_{s_{0}}^{t}e^{-(q+1)(t-s)}\int_{\Omega}u^{q+1}ds+C_{4}(\varepsilon_{1},\varepsilon_{2},q),
\end{aligned}
\end{equation}
where $s_{0}$ is the same as in \eqref{s0} and a positive constant
\begin{equation*}
C_{4}(\varepsilon_{1},\varepsilon_{2},q)=\frac{1}{q}\|u(\cdot, s_{0})\|_{L^{q}(\Omega)}^{q}+\frac{C_{3}(\varepsilon_{1},\varepsilon_{2},q)}{q+1}.
\end{equation*}
Applying Lemma \ref{b}, we can estimate
\begin{equation}\label{12}
\begin{aligned}
&A_{1}\lambda_{0}^{-q}\chi^{q+1}\int_{s_{0}}^{t}e^{-(q+1)(t-s)}\int_{\Omega}\left(|\Delta w|^{q+1}+w^{q+1}\right)ds\\
\leq  &A_{1}\lambda_{0}^{-q}\chi^{q+1}e^{-(q+1)t}C_{q+1}\int_{s_{0}}^{t}e^{(q+1)s}\|u+v\|_{L^{q+1}(\Omega)}^{q+1}\\
+&A_{1}\lambda_{0}^{-q}\chi^{q+1}e^{-(q+1)(t-s_{0})}C_{q+1}\left(\|w(\cdot,s_{0})\|_{L^{q+1}(\Omega)}^{q+1}+\|\Delta w(\cdot,s_{0})\|_{L^{q+1}(\Omega)}^{q+1}\right)\\
\leq  &A_{1}\lambda_{0}^{-q}\chi^{q+1}C_{q+1}2^{q}\left(\int_{s_{0}}^{t}e^{-(q+1)(t-s)}\|u(\cdot,s)\|_{L^{q+1}(\Omega)}^{q+1}
+\int_{s_{0}}^{t}e^{-(q+1)(t-s)}\|v(\cdot,s)\|_{L^{q+1}(\Omega)}^{q+1}\right)\\
+&A_{1}\lambda_{0}^{-q}\chi^{q+1}e^{-(q+1)(t-s_{0})}C_{q+1}\left(\|w(\cdot,s_{0})\|_{L^{q+1}(\Omega)}^{q+1}+\|\Delta w(\cdot,s_{0})\|_{L^{q+1}(\Omega)}^{q+1}\right)
\end{aligned}
\end{equation}
for all $t\in(s_{0},T_{max})$. Then the combination of \eqref{12} and \eqref{10} gives
\begin{equation}\label{13}
\begin{aligned}
&\frac{1}{q}\|u(\cdot, t)\|_{L^{q}(\Omega)}^{q}\\
\leq  &A_{1}\lambda_{0}^{-q}\chi^{q+1}C_{q+1}2^{q}\left(\int_{s_{0}}^{t}e^{-(q+1)(t-s)}\|u(\cdot,s)\|_{L^{q+1}(\Omega)}^{q+1}
+\int_{s_{0}}^{t}e^{-(q+1)(t-s)}\|v(\cdot,s)\|_{L^{q+1}(\Omega)}^{q+1}\right)\\
+&(\lambda_{0}+\varepsilon_{1}+\varepsilon_{2}-\mu_{1})\int_{s_{0}}^{t}e^{-(q+1)(t-s)}\int_{\Omega}u^{q+1}ds+C_{5}(\varepsilon_{1},\varepsilon_{2},q) \ \ \ \mbox{for all}\ t \in(s_{0},T_{max}).
\end{aligned}
\end{equation}
where the positive constant
\begin{equation*}
\begin{aligned}
&C_{5}(\varepsilon_{1},\varepsilon_{2},q)\\
=&A_{1}\lambda_{0}^{-q}\chi^{q+1}C_{q+1}\left(\|w(\cdot,s_{0})\|_{L^{q+1}(\Omega)}^{q+1}+\|\Delta w(\cdot,s_{0})\|_{L^{q+1}(\Omega)}^{q+1}\right)+C_{4}(\varepsilon_{1},\varepsilon_{2},q).
\end{aligned}
\end{equation*}
Similarly, testing the second equation against $v^{q-1}$ and integrating by parts, we infer that
\begin{equation*}
\begin{aligned}
&\frac{1}{q}\frac{d}{dt}\int_{\Omega}v^{q}+\frac{q+1}{q}\int_{\Omega}v^{q}\\
\leq&\frac{q-1}{q}\chi\int_{\Omega}v^{q}|\Delta w|+\int_{\Omega}v^{q-1}w+r\int_{\Omega}uv^{q}+\frac{q+1}{q}\int_{\Omega}v^{q}-\mu_{2}\int_{\Omega}v^{q+1} \ \ \ \mbox{for all}\ t \in(0,T_{max}).
\end{aligned}
\end{equation*}
The term $r\int_{\Omega}uv^{q}$ follows from Young's inequality that
\begin{equation*}
r\int_{\Omega}uv^{q}\leq r\left(\frac{1}{q}\right)^{\frac{1}{q+1}}\frac{q}{q+1}\int_{\Omega}v^{q+1}+r\left(\frac{1}{q}\right)^{\frac{1}{q+1}}\frac{q}{q+1}\int_{\Omega}u^{q+1} \ \ \ \mbox{for all}\ t \in(0,T_{max}).
\end{equation*}
Repeating the same process from \eqref{3} to \eqref{4}, for any $\varepsilon_{3}>0$ and $\varepsilon_{4}>0$, there exists a positive constant $C_{6}(\varepsilon_{3},\varepsilon_{4},q)$ such that
\begin{equation}\label{17}
\begin{aligned}
&\frac{1}{q}\frac{d}{dt}\int_{\Omega}v^{q}+\frac{q+1}{q}\int_{\Omega}v^{q}\\
\leq &A_{1}\lambda_{0}^{-q}\chi^{q+1}\left(\int_{\Omega}|\Delta w|^{q+1}+\int_{\Omega}w^{q+1}\right)+r\left(\frac{1}{q}\right)^{\frac{1}{q+1}}\frac{q}{q+1}\int_{\Omega}u^{q+1} \\
+&\left[\lambda_{0}+\varepsilon_{3}+\varepsilon_{4}+r\left(\frac{1}{q}\right)^{\frac{1}{q+1}}\frac{q}{q+1}-\mu_{2}\right]\int_{\Omega}v^{q+1}+C_{6}(\varepsilon_{3},\varepsilon_{4},q)\ \ \ \mbox{for all}\ t \in(0,T_{max}).
\end{aligned}
\end{equation}
Thanks to the idea in  \eqref{9}-\eqref{13}, \eqref{17} yields a positive constant $C_{7}(\varepsilon_{3},\varepsilon_{4},q)$ such that
\begin{equation}\label{18}
\begin{aligned}
&\frac{1}{q}\|v(\cdot, t)\|_{L^{q}(\Omega)}^{q}\\
\leq  &A_{1}\lambda_{0}^{-q}\chi^{q+1}C_{q+1}2^{q}\left(\int_{s_{0}}^{t}e^{-(q+1)(t-s)}\|u(\cdot,s)\|_{L^{q+1}(\Omega)}^{q+1}
+\int_{s_{0}}^{t}e^{-(q+1)(t-s)}\|v(\cdot,s)\|_{L^{q+1}(\Omega)}^{q+1}\right)\\
+&\left[\lambda_{0}+\varepsilon_{3}+\varepsilon_{4}+r\left(\frac{1}{q}\right)^{\frac{1}{q+1}}\frac{q}{q+1}-\mu_{2}\right]\int_{s_{0}}^{t}e^{-(q+1)(t-s)}\int_{\Omega}v^{q+1}ds\\
+&r\left(\frac{1}{q}\right)^{\frac{1}{q+1}}\frac{q}{q+1}\int_{s_{0}}^{t}e^{-(q+1)(t-s)}\int_{\Omega}u^{q+1}ds+C_{7}(\varepsilon_{3},\varepsilon_{4},q) \ \ \ \mbox{for all}\ t \in(s_{0},T_{max}).
\end{aligned}
\end{equation}
Based on Lemma \ref{a}, \eqref{13} and \eqref{18} imply that
\begin{equation}\label{19}
\begin{aligned}
&\frac{1}{q}\left(\|u(\cdot, t)\|_{L^{q}(\Omega)}^{q}+\|v(\cdot, t)\|_{L^{q}(\Omega)}^{q}\right)-C_{8}(\varepsilon_{1},\varepsilon_{2},\varepsilon_{3},\varepsilon_{4},q)\\
\leq  &\left[A_{1}\lambda_{0}^{-q}\chi^{q+1}C_{q+1}2^{q+1}+\lambda_{0}+\varepsilon_{1}+\varepsilon_{2}+r\left(\frac{1}{q}\right)^{\frac{1}{q+1}}\frac{q}{q+1}-\mu_{1}\right]\int_{s_{0}}^{t}e^{-(q+1)(t-s)}\|u(\cdot,s)\|_{L^{q+1}(\Omega)}^{q+1}\\
+&\left[A_{1}\lambda_{0}^{-q}\chi^{q+1}C_{q+1}2^{q+1}+\lambda_{0}+\varepsilon_{3}+\varepsilon_{4}+r\left(\frac{1}{q}\right)^{\frac{1}{q+1}}\frac{q}{q+1}-\mu_{2}\right]\int_{s_{0}}^{t}e^{-(q+1)(t-s)}\|v(\cdot,s)\|_{L^{q+1}(\Omega)}^{q+1}\\
=&\left[\frac{2(q-1)}{q}C_{q+1}^{\frac{1}{q+1}}\chi+\varepsilon_{1}+\varepsilon_{2}+r\left(\frac{1}{q}\right)^{\frac{1}{q+1}}\frac{q}{q+1}-\mu_{1}\right]\int_{s_{0}}^{t}e^{-(q+1)(t-s)}\|u(\cdot,s)\|_{L^{q+1}(\Omega)}^{q+1}\\
+&\left[\frac{2(q-1)}{q}C_{q+1}^{\frac{1}{q+1}}\chi+\varepsilon_{3}+\varepsilon_{4}+r\left(\frac{1}{q}\right)^{\frac{1}{q+1}}\frac{q}{q+1}-\mu_{2}\right]\int_{s_{0}}^{t}e^{-(q+1)(t-s)}\|v(\cdot,s)\|_{L^{q+1}(\Omega)}^{q+1}\\
\end{aligned}
\end{equation}
for all $t\in(s_{0},T_{max})$ with some positive constant $C_{8}(\varepsilon_{1},\varepsilon_{2},\varepsilon_{3},\varepsilon_{4},q)=C_{5}(\varepsilon_{1},\varepsilon_{2},q)+C_{7}(\varepsilon_{3},\varepsilon_{4},q)$. Since
\begin{equation*}
\mu_{i}>\frac{2(N-2)_{+}}{N}C_{\frac{N}{2}+1}^{\frac{1}{\frac{N}{2}+1}}\chi+\left[(\frac{2}{N})^{\frac{2}{N+2}}\frac{N}{N+2}\right]r \ \ \  (i=1,2),
\end{equation*}
we can choose $q:=q_{0}>\frac{N}{2}$ such that
\begin{equation*}
\mu_{i}>\frac{2(q_{0}-1)}{q_{0}}C_{q_{0}+1}^{\frac{1}{q_{0}+1}}\chi+r\left(\frac{1}{q_{0}}\right)^{\frac{1}{q_{0}+1}}\frac{q_{0}}{q_{0}+1} \ \ \  (i=1,2).
\end{equation*}
Then we pick $\varepsilon_{i}$ $(i=1, .., 4)$ that are sufficiently small such that
\begin{equation*}
0<\varepsilon_{1}+\varepsilon_{2}<\mu_{1}-\frac{2(q_{0}-1)}{q_{0}}C_{q_{0}+1}^{\frac{1}{q_{0}+1}}\chi-r\left(\frac{1}{q_{0}}\right)^{\frac{1}{q_{0}+1}}\frac{q_{0}}{q_{0}+1}
\end{equation*}
and
\begin{equation*}
0<\varepsilon_{3}+\varepsilon_{4}<\mu_{2}-\frac{2(q_{0}-1)}{q_{0}}C_{q_{0}+1}^{\frac{1}{q_{0}+1}}\chi-r\left(\frac{1}{q_{0}}\right)^{\frac{1}{q_{0}+1}}\frac{q_{0}}{q_{0}+1}.
\end{equation*}
Therefore, it follows from \eqref{s0} and \eqref{19} that \eqref{Lp0} holds.
\end{proof}
\begin{lemma}\label{lemmaLp}
Let $N\geq3$ and $r_{1}=r_{2}=2$. Then for any $p>1$ there exists a constant $C>0$ such that if $\min\{\mu_{1},\mu_{2}\}>\frac{2(N-2)_{+}}{N}C_{\frac{N}{2}+1}^{\frac{1}{\frac{N}{2}+1}}\max\{\chi_{1},\chi_{2}\}+\left[(\frac{2}{N})^{\frac{2}{N+2}}\frac{N}{N+2}\right]r$, one has
\begin{equation}\label{Lp}
\|u(\cdot, t)\|_{L^{p}(\Omega)}+\|v(\cdot, t)\|_{L^{p}(\Omega)}\leq C  \ \ \ \mbox{for all}\ t \in(0,T_{max}).
\end{equation}
\end{lemma}

\begin{proof}
By Lemma \ref{d}, \eqref{Lp0} shows that
\begin{equation}\label{25}
\| w(\cdot, t)\|_{W^{1,r}(\Omega)}\leq C_{1}
\end{equation}
for all $t \in(0,T_{max})$ and $r\in[1,\frac{Nq_{0}}{(N-q_{0})^{+}})$ with $C_{1}>0$. Applying the Sobolev embedding theorem, we can find a constant $C_{2}>0$ such that
\begin{equation}\label{26}
\| w(\cdot, t)\|_{L^{\infty}(\Omega)}\leq C_{2}\ \ \ \mbox{for all}\ t \in(0,T_{max}).
\end{equation}
Multiplying the first equation \eqref{model} by $u^{p-1}$ $(p>1)$, we integrate by parts to obtain
\begin{equation}\label{27}
\frac{1}{p}\frac{d}{dt}\int_{\Omega}u^{p}+(p-1)\int_{\Omega}u^{p-2}|\nabla u|^{2}=\chi_{1}(p-1)\int_{\Omega}u^{p-1}\nabla u\cdot\nabla w+\int_{\Omega}u^{p-1}w-\mu_{1}\int_{\Omega}u^{p+1},
\end{equation}
where Young's inequality and \eqref{26} give
\begin{equation}\label{28}
\chi_{1}(p-1)\int_{\Omega}u^{p-1}\nabla u\cdot\nabla w\leq \frac{p-1}{2}\int_{\Omega}u^{p-2}|\nabla u|^{2}+\frac{\chi_{1}^{2}(p-1)}{2}\int_{\Omega}u^{p}|\nabla w|^{2}
\end{equation}
and
\begin{equation}\label{29}
\begin{aligned}
\int_{\Omega}u^{p-1}w\leq &\frac{\mu_{1}}{2}\int_{\Omega}u^{p+1}+C_{3}\int_{\Omega}w^{\frac{p+1}{2}}\\
\leq&\frac{\mu_{1}}{2}\int_{\Omega}u^{p+1}+C_{4}
\end{aligned}
\end{equation}
with some positive constants $C_{3}$ and $C_{4}$. Substituting \eqref{28} and \eqref{29} into \eqref{27}, we have
\begin{equation}\label{30}
\frac{1}{p}\frac{d}{dt}\int_{\Omega}u^{p}+\frac{p-1}{2}\int_{\Omega}u^{p-2}|\nabla u|^{2}\leq\frac{\chi^{2}(p-1)}{2}\int_{\Omega}u^{p}|\nabla w|^{2}-\frac{\mu_{1}}{2}\int_{\Omega}u^{p+1}+C_{4}
\end{equation}
for all $t\in(0,T_{max})$. In view of the H\"{o}lder inequality and \eqref{25}, there exists a constant $C_{5}>0$ such that
\begin{equation}\label{31}
\begin{aligned}
\frac{\chi_{1}^{2}(p-1)}{2}\int_{\Omega}u^{p}|\nabla w|^{2}\leq&\frac{\chi^{2}(p-1)}{2}\left(\int_{\Omega}u^{\frac{pq_{0}}{q_{0}-1}}\right)^{\frac{q_{0}-1}{q_{0}}}\left(\int_{\Omega}|\nabla w|^{2q_{0}}\right)^{\frac{1}{q_{0}}}\\
\leq&C_{5}\|u^{\frac{p}{2}}\|_{L^{\frac{2q_{0}}{q_{0}-1}}(\Omega)}^{2}\ \ \ \mbox{for all}\ t \in(0,T_{max})
\end{aligned}
\end{equation}
thanks to $2q_{0}<\frac{Nq_{0}}{(N-q_{0})^{+}}$, where $q_{0}>\frac{N}{2}$ ($N\geq3$) coincides with that in \eqref{Lp0}. Letting $p>q_{0}+1$, the fact $q_{0}>\frac{N}{2}$ yields that
\begin{equation*}
\frac{q_{0}}{p}<\frac{q_{0}}{q_{0}-1}<\frac{N}{N-2}.
\end{equation*}
Due to the Gagliardo-Nirenberg inequality, for some positive constants $C_{6}$ and $C_{7}$ we conclude that
\begin{equation*}
\begin{aligned}
C_{5}\|u^{\frac{p}{2}}\|_{L^{\frac{2q_{0}}{q_{0}-1}}(\Omega)}^{2}\leq& C_{6}\left(\|\nabla u^{\frac{p}{2}}\|^{2\theta}_{L^{2}(\Omega)}\|u^{\frac{p}{2}}\|^{2(1-\theta)}_{L^{\frac{2q_{0}}{p}}(\Omega)}+\|u^{\frac{p}{2}}\|^{2}_{L^{\frac{2q_{0}}{p}}(\Omega)}\right)\\
\leq&C_{7}\left(\|\nabla u^{\frac{p}{2}}\|^{2\theta}_{L^{2}(\Omega)}+1\right)\ \ \ \mbox{for all}\ t \in(0,T_{max}),
\end{aligned}
\end{equation*}
where
\begin{equation*}
\theta=\frac{\frac{Np}{2q_{0}}-\frac{Np}{2\frac{q_{0}}{q_{0}-1}p}}{1-\frac{N}{2}+\frac{Np}{2q_{0}}}\in(0,1).
\end{equation*}
Since $\theta<1$, we may employ Young's inequality to estimate
\begin{equation*}
C_{5}\|u^{\frac{p}{2}}\|_{L^{\frac{2q_{0}}{q_{0}-1}}(\Omega)}^{2}\leq \frac{p-1}{p^{2}}\int_{\Omega}|\nabla u^{\frac{p}{2}}|^{2}+C_{8}\ \ \ \mbox{for all}\ t \in(0,T_{max})
\end{equation*}
with some constant $C_{8}>0$, which updates \eqref{31} as
\begin{equation}\label{36}
\frac{\chi_{1}^{2}(p-1)}{2}\int_{\Omega}u^{p}|\nabla w|^{2}\leq\frac{p-1}{p^{2}}\int_{\Omega}|\nabla u^{\frac{p}{2}}|^{2}+C_{8}\ \ \ \mbox{for all}\ t \in(0,T_{max}).
\end{equation}
Noting the fact $\int_{\Omega}u^{p-2}|\nabla u|^{2}=\frac{4}{p^{2}}\int_{\Omega}|\nabla u^{\frac{p}{2}}|^{2}$, the combination of \eqref{30} and \eqref{36} entails
\begin{equation*}
\begin{aligned}
\frac{1}{p}\frac{d}{dt}\int_{\Omega}u^{p}+\frac{p-1}{p^{2}}\int_{\Omega}|\nabla u^{\frac{p}{2}}|^{2}\leq&-\frac{\mu_{1}}{2}\int_{\Omega}u^{p+1}+C_{9}\\
\leq&-\frac{\mu_{1}}{2|\Omega|^{\frac{1}{p}}}\left(\int_{\Omega}u^{p}\right)^{\frac{p+1}{p}}+C_{9}
\end{aligned}
\end{equation*}
with some positive constant $C_{9}=C_{4}+C_{8}$. Consequently, there exists a positive constant $C_{10}=\frac{p\mu_{1}}{2|\Omega|^{\frac{1}{p}}}$ such that $y(t):=\int_{\Omega}u^{p}(t)$ satisfies
\begin{equation}\label{38}
y'(t)+C_{10}y^{\frac{p+1}{p}}\leq pC_{9}\ \ \ \mbox{for all}\ t \in(0,T_{max}),
\end{equation}
which from a standard ODE comparison argument yields
\begin{equation}\label{39}
\|u(\cdot,t)\|_{L^{p}(\Omega)}\leq C_{11}\ \ \ \mbox{for all}\ t \in(0,T_{max})
\end{equation}
with some constant $C_{11}>0$.

As for the component $v$, by a straightforward testing procedure, we see that
\begin{equation}\label{40}
\begin{aligned}
\frac{1}{p}\frac{d}{dt}\int_{\Omega}v^{p}+(p-1)\int_{\Omega}v^{p-2}|\nabla v|^{2}=&\chi_{2}(p-1)\int_{\Omega}v^{p-1}\nabla v\cdot\nabla w+\int_{\Omega}v^{p-1}w\\
&+r\int_{\Omega}uv^{p}-\mu_{2}\int_{\Omega}v^{p+1}\ \ \ \mbox{for all}\ t \in(0,T_{max}).
\end{aligned}
\end{equation}
Compared with \eqref{27}, the only difference is the presence of the term $r\int_{\Omega}uv^{p}$. Based on \eqref{39}, for some positive constants $C_{12}$ and $C_{13}$ we use Young's inequality to derive
\begin{equation}\label{41}
\begin{aligned}
r\int_{\Omega}uv^{p}\leq& \frac{\mu_{2}}{2}\int_{\Omega}v^{p+1}+C_{12}\int_{\Omega}u^{p+1}\\
\leq&\frac{\mu_{2}}{2}\int_{\Omega}v^{p+1}+C_{13}\ \ \ \mbox{for all}\ t \in(0,T_{max}).
\end{aligned}
\end{equation}
Similarly, inserting \eqref{41} into \eqref{40} and repeating the process from \eqref{28} to \eqref{38}, we can find some constant $C_{14}>0$ such that
\begin{equation}\label{42}
\|v(\cdot,t)\|_{L^{p}(\Omega)}\leq C_{14}\ \ \ \mbox{for all}\ t \in(0,T_{max}),
\end{equation}
which along with \eqref{39} proves \eqref{Lp}.
\end{proof}

\begin{lemma}\label{p-q}
Let $N\geq3$ and $r_{i}>2$ $(i=1,2)$. Then for any $p>1$ and $q>1$, one can find a positive constant $C$ such that
\begin{equation}\label{pq}
\|u(\cdot,t)\|_{L^{p}(\Omega)}^{p}+\|v(\cdot,t)\|_{L^{q}(\Omega)}^{q}\leq C\ \ \ \mbox{for all}\ t \in(0,T_{max}).
\end{equation}
\end{lemma}
\begin{proof}
We test the first equation in \eqref{model} against $u^{p-1}$ ($p>1$) and integrate by parts to see that
\begin{equation}\label{54}
\begin{aligned}
&\frac{1}{p}\frac{d}{dt}\int_{\Omega}u^{p}+(p-1)\int_{\Omega}u^{p-2}|\nabla u|^{2}\\
\leq& \frac{p-1}{p}\chi_{1}\int_{\Omega}u^{p}|\Delta w|+\int_{\Omega}u^{p-1}w-\mu_{1}\int_{\Omega}u^{p+r_{1}-1}\ \ \ \mbox{for all}\ t \in(0,T_{max}).
\end{aligned}
\end{equation}
Here by Young's inequality, for some positive constants $C_{i}$ $(i=1, .., 3)$ we obtain
\begin{equation}\label{55}
\frac{p-1}{p}\chi_{1}\int_{\Omega}u^{p}|\Delta w|\leq \frac{\mu_{1}}{4}\int_{\Omega}u^{p+r_{1}-1}+C_{1}\int_{\Omega}|\Delta w|^{\frac{p+r_{1}-1}{r_{1}-1}}
\end{equation}
and
\begin{equation}\label{56}
\begin{aligned}
\int_{\Omega}u^{p-1}w\leq& \frac{\mu_{1}}{4}\int_{\Omega}u^{p+r_{1}-1}+C_{2}\int_{\Omega} w^{\frac{p+r_{1}-1}{r_{1}}}\\
\leq& C_{1}\int_{\Omega}w^{\frac{p+r_{1}-1}{r_{1}-1}}+\frac{\mu_{1}}{4}\int_{\Omega}u^{p+r_{1}-1}+C_{3}.
\end{aligned}
\end{equation}
Then the combination of \eqref{55}-\eqref{56} and \eqref{54} gives
\begin{equation}\label{57}
\begin{aligned}
&\frac{d}{dt}\int_{\Omega}u^{p}+\frac{p+r_{1}-1}{r_{1}-1}\int_{\Omega}u^{p}\\
\leq& C_{1}p\left(\int_{\Omega}|\Delta w|^{\frac{p+r_{1}-1}{r_{1}-1}}+\int_{\Omega}w^{\frac{p+r_{1}-1}{r_{1}-1}}\right)+\frac{p+r_{1}-1}{r_{1}-1}\int_{\Omega}u^{p}-\frac{\mu_{1}p}{2}\int_{\Omega}u^{p+r_{1}-1}+C_{3}p\\
\leq& C_{1}p\left(\int_{\Omega}|\Delta w|^{\frac{p+r_{1}-1}{r_{1}-1}}+\int_{\Omega}w^{\frac{p+r_{1}-1}{r_{1}-1}}\right)-\frac{\mu_{1}p}{4}\int_{\Omega}u^{p+r_{1}-1}+C_{4}\ \ \ \mbox{for all}\ t \in(0,T_{max})
\end{aligned}
\end{equation}
with some constant $C_{4}>0$, where we have used Young's inequality. Hence for all $t\in(0,T_{max})$, \eqref{57} can be rewritten as
\begin{equation}\label{58}
\begin{aligned}
&\frac{d}{dt}\left(e^{\frac{p+r_{1}-1}{r_{1}-1}t}\|u(\cdot,t)\|_{L^{p}(\Omega)}^{p}\right)\\
\leq& \left[C_{1}p\left(\int_{\Omega}|\Delta w|^{\frac{p+r_{1}-1}{r_{1}-1}}+\int_{\Omega}w^{\frac{p+r_{1}-1}{r_{1}-1}}\right)-\frac{\mu_{1}p}{4}\int_{\Omega}u^{p+r_{1}-1}+C_{4}\right]e^{\frac{p+r_{1}-1}{r_{1}-1}t}.
\end{aligned}
\end{equation}
Let $s_{0}$ be the same as in \eqref{s0}. Integrating \eqref{58} over $[s_{0},t)$, by means of Lemma \ref{b} and Young's inequality we can find positive constants $C_{i}$ $(i=5, ..., 7)$ and $C_{\ast}$ such that for all $t\in(s_{0},T_{max})$,
\begin{equation}\label{59}
\begin{aligned}
&\|u(\cdot, t)\|_{L^{p}(\Omega)}^{p}\\
\leq &e^{-\frac{p+r_{1}-1}{r_{1}-1}(t-s_{0})}\|u(\cdot, s_{0})\|_{L^{p}(\Omega)}^{p}+C_{1}pe^{-\frac{p+r_{1}-1}{r_{1}-1}t}\int_{s_{0}}^{t}e^{\frac{p+r_{1}-1}{r_{1}-1}s}\int_{\Omega}\left(|\Delta w|^{\frac{p+r_{1}-1}{r_{1}-1}}+w^{\frac{p+r_{1}-1}{r_{1}-1}}\right)ds\\
&+C_{4}\int_{s_{0}}^{t}e^{-\frac{p+r_{1}-1}{r_{1}-1}(t-s)}ds-\frac{\mu_{1}p}{4}e^{-\frac{p+r_{1}-1}{r_{1}-1}t}\int_{s_{0}}^{t}e^{\frac{p+r_{1}-1}{r_{1}-1}s}\int_{\Omega}u^{p+r_{1}-1}ds\\
\leq  &C_{1}C_{\ast}pe^{-\frac{p+r_{1}-1}{r_{1}-1}t}\int_{s_{0}}^{t}e^{\frac{p+r_{1}-1}{r_{1}-1}s}\|u+v\|_{L^{\frac{p+r_{1}-1}{r_{1}-1}}(\Omega)}^{\frac{p+r_{1}-1}{r_{1}-1}}ds-\frac{\mu_{1}p}{4}e^{-\frac{p+r_{1}-1}{r_{1}-1}t}\int_{s_{0}}^{t}e^{\frac{p+r_{1}-1}{r_{1}-1}s}\int_{\Omega}u^{p+r_{1}-1}ds\\
&+C_{1}C_{\ast}pe^{-\frac{p+r_{1}-1}{r_{1}-1}(t-s_{0})}\left(\|w(\cdot,s_{0})\|_{L^{\frac{p+r_{1}-1}{r_{1}-1}}(\Omega)}^{\frac{p+r_{1}-1}{r_{1}-1}}+\|\Delta w(\cdot,s_{0})\|_{L^{\frac{p+r_{1}-1}{r_{1}-1}}(\Omega)}^{\frac{p+r_{1}-1}{r_{1}-1}}\right)+C_{5}\\
\leq  &C_{1}C_{\ast}p2^{\frac{p}{r_{1}-1}}\left(\int_{s_{0}}^{t}e^{-\frac{p+r_{1}-1}{r_{1}-1}(t-s)}\|u(\cdot,s)\|_{L^{\frac{p+r_{1}-1}{r_{1}-1}}(\Omega)}^{\frac{p+r_{1}-1}{r_{1}-1}}ds
+\int_{s_{0}}^{t}e^{-\frac{p+r_{1}-1}{r_{1}-1}(t-s)}\|v(\cdot,s)\|_{L^{\frac{p+r_{1}-1}{r_{1}-1}}(\Omega)}^{\frac{p+r_{1}-1}{r_{1}-1}}ds\right)\\
&-\frac{\mu_{1}p}{4}e^{-\frac{p+r_{1}-1}{r_{1}-1}t}\int_{s_{0}}^{t}e^{\frac{p+r_{1}-1}{r_{1}-1}s}\int_{\Omega}u^{p+r_{1}-1}ds+C_{6}\\
\leq&C_{1}C_{\ast}p2^{\frac{p}{r_{1}-1}}\int_{s_{0}}^{t}e^{-\frac{p+r_{1}-1}{r_{1}-1}(t-s)}\|v(\cdot,s)\|_{L^{\frac{p+r_{1}-1}{r_{1}-1}}(\Omega)}^{\frac{p+r_{1}-1}{r_{1}-1}}ds-\frac{\mu_{1}p}{8}\int_{s_{0}}^{t}e^{-\frac{p+r_{1}-1}{r_{1}-1}(t-s)}\int_{\Omega}u^{p+r_{1}-1}ds+C_{7}.
\end{aligned}
\end{equation}

Likewise, multiplying the second equation by $v^{q-1}$ $(q>1)$, some straightforward computations on the basis of integration by parts and Young's inequality show that
\begin{equation}\label{60}
\begin{aligned}
&\frac{d}{dt}\int_{\Omega}v^{q}+\frac{q+r_{2}-1}{r_{2}-1}\int_{\Omega}v^{q}\\
\leq& C_{8}q\left(\int_{\Omega}|\Delta w|^{\frac{q+r_{2}-1}{r_{2}-1}}+\int_{\Omega}w^{\frac{q+r_{2}-1}{r_{2}-1}}\right)+rq\int_{\Omega}uv^{q}-\frac{\mu_{2}q}{4}\int_{\Omega}v^{q+r_{2}-1}+C_{9}\ \ \ \mbox{for all}\ t \in(0,T_{max})
\end{aligned}
\end{equation}
with some positive constants $C_{8}$ and $C_{9}$. Then for all $t\in(0,T_{max}$), we further employ Young's inequality to pick positive constants $C_{10}$ such that
\begin{equation*}
r\int_{\Omega}uv^{q}\leq \frac{\mu_{2}}{8}\int_{\Omega}v^{q+r_{2}-1}+C_{10}\int_{\Omega}u^{\frac{q+r_{2}-1}{r_{2}-1}},
\end{equation*}
which updates \eqref{60} as
\begin{equation}\label{63}
\begin{aligned}
&\frac{d}{dt}\int_{\Omega}v^{q}+\frac{q+r_{2}-1}{r_{2}-1}\int_{\Omega}v^{q}\\
\leq& C_{8}q\left(\int_{\Omega}|\Delta w|^{\frac{q+r_{2}-1}{r_{2}-1}}+\int_{\Omega}w^{\frac{q+r_{2}-1}{r_{2}-1}}\right)+C_{10}q\int_{\Omega}u^{\frac{q+r_{2}-1}{r_{2}-1}}-\frac{\mu_{2}q}{8}\int_{\Omega}v^{q+r_{2}-1}+C_{9}.
\end{aligned}
\end{equation}
Since $r_{i}>2$ $(i=1,2)$, we can choose suitably large $p$ and $q$ fulfilling both
\begin{equation}\label{64}
\frac{q+r_{2}-1}{r_{2}-1}<p+r_{1}-1
\end{equation}
and
\begin{equation}\label{65}
\frac{p+r_{1}-1}{r_{1}-1}<q+r_{2}-1.
\end{equation}
Therefore, integrating \eqref{63} over $[s_{0},t)$, Lemma \ref{b} and Young's inequality provide some positive constants $C_{i}$ $(i=11 ,..., 13)$ and $C_{\ast\ast}$ such that
\begin{equation}\label{66}
\begin{aligned}
&\|v(\cdot, t)\|_{L^{q}(\Omega)}^{q}+\frac{\mu_{2}q}{8}e^{-\frac{q+r_{2}-1}{r_{2}-1}t}\int_{s_{0}}^{t}e^{\frac{q+r_{2}-1}{r_{2}-1}s}\int_{\Omega}v^{q+r_{2}-1}ds\\
\leq &e^{-\frac{q+r_{2}-1}{r_{2}-1}(t-s_{0})}\|v(\cdot, s_{0})\|_{L^{q}(\Omega)}^{q}+C_{8}qe^{-\frac{q+r_{2}-1}{r_{2}-1}t}\int_{s_{0}}^{t}e^{\frac{q+r_{2}-1}{r_{2}-1}s}\int_{\Omega}\left(|\Delta w|^{\frac{q+r_{2}-1}{r_{2}-1}}+w^{\frac{q+r_{2}-1}{r_{2}-1}}\right)ds\\
&+C_{10}qe^{-\frac{q+r_{2}-1}{r_{2}-1}t}\int_{s_{0}}^{t}e^{\frac{q+r_{2}-1}{r_{2}-1}s}\int_{\Omega}u^{\frac{q+r_{2}-1}{r_{2}-1}}ds+C_{9}\int_{s_{0}}^{t}e^{-\frac{q+r_{2}-1}{r_{2}-1}(t-s)}ds\\
\leq  &C_{8}C_{\ast\ast}qe^{-\frac{q+r_{2}-1}{r_{2}-1}t}\int_{s_{0}}^{t}e^{\frac{q+r_{2}-1}{r_{2}-1}s}\|u+v\|_{L^{\frac{q+r_{2}-1}{r_{2}-1}}(\Omega)}^{\frac{q+r_{2}-1}{r_{2}-1}}ds+C_{10}qe^{-\frac{q+r_{2}-1}{r_{2}-1}t}\int_{s_{0}}^{t}e^{\frac{q+r_{2}-1}{r_{2}-1}s}\int_{\Omega}u^{\frac{q+r_{2}-1}{r_{2}-1}}ds\\
&+C_{8}C_{\ast\ast}qe^{-\frac{q+r_{2}-1}{r_{2}-1}(t-s_{0})}\left(\|w(\cdot,s_{0})\|_{L^{\frac{q+r_{2}-1}{r_{2}-1}}(\Omega)}^{\frac{q+r_{2}-1}{r_{2}-1}}+\|\Delta w(\cdot,s_{0})\|_{L^{\frac{q+r_{2}-1}{r_{2}-1}}(\Omega)}^{\frac{q+r_{2}-1}{r_{2}-1}}\right)+C_{11}\\
\leq  &C_{8}C_{\ast\ast}q2^{\frac{q}{r_{2}-1}}\left(\int_{s_{0}}^{t}e^{-\frac{q+r_{2}-1}{r_{2}-1}(t-s)}\|u(\cdot,s)\|_{L^{\frac{q+r_{2}-1}{r_{2}-1}}(\Omega)}^{\frac{q+r_{2}-1}{r_{2}-1}}ds
+\int_{s_{0}}^{t}e^{-\frac{q+r_{2}-1}{r_{2}-1}(t-s)}\|v(\cdot,s)\|_{L^{\frac{q+r_{2}-1}{r_{2}-1}}(\Omega)}^{\frac{q+r_{2}-1}{r_{2}-1}}ds\right)\\
&+C_{10}q\int_{s_{0}}^{t}e^{-\frac{q+r_{2}-1}{r_{2}-1}(t-s)}\|u(\cdot,s)\|_{L^{\frac{q+r_{2}-1}{r_{2}-1}}(\Omega)}^{\frac{q+r_{2}-1}{r_{2}-1}}+C_{12}\\
\leq  &\left(C_{8}C_{\ast\ast}q2^{\frac{q}{r_{2}-1}}+C_{10}q\right)\int_{s_{0}}^{t}e^{-\frac{q+r_{2}-1}{r_{2}-1}(t-s)}\|u(\cdot,s)\|_{L^{\frac{q+r_{2}-1}{r_{2}-1}}(\Omega)}^{\frac{q+r_{2}-1}{r_{2}-1}}ds\\
+&\frac{\mu_{2}q}{16}e^{-\frac{q+r_{2}-1}{r_{2}-1}t}\int_{s_{0}}^{t}e^{\frac{q+r_{2}-1}{r_{2}-1}s}\int_{\Omega}v^{q+r_{2}-1}ds+C_{13}\ \ \ \mbox{for all}\ t \in(s_{0},T_{max}).
\end{aligned}
\end{equation}
Combining \eqref{59} with \eqref{66} and recalling the facts \eqref{64}-\eqref{65}, we use Young's inequality to see that there exist positive constants $C_{14}=\max\{C_{8}C_{\ast\ast}q2^{\frac{q}{r_{2}-1}}+C_{10}q,C_{1}C_{\ast}p2^{\frac{p}{r_{1}-1}}\}$, $C_{15}=C_{7}+C_{13}$ and $C_{16}$ fulfilling
\begin{equation*}
\begin{aligned}
&\|u(\cdot, t)\|_{L^{p}(\Omega)}^{p}+\|v(\cdot, t)\|_{L^{q}(\Omega)}^{q}\\
\leq  &C_{14}\left(\int_{s_{0}}^{t}e^{-\frac{q+r_{2}-1}{r_{2}-1}(t-s)}\|u(\cdot,s)\|_{L^{\frac{q+r_{2}-1}{r_{2}-1}}(\Omega)}^{\frac{q+r_{2}-1}{r_{2}-1}}ds+\int_{s_{0}}^{t}e^{-\frac{p+r_{1}-1}{r_{1}-1}(t-s)}\|v(\cdot,s)\|_{L^{\frac{p+r_{1}-1}{r_{1}-1}}(\Omega)}^{\frac{p+r_{1}-1}{r_{1}-1}}ds\right)
\\
&-\frac{\mu_{1}p}{8}\int_{s_{0}}^{t}e^{-\frac{p+r_{1}-1}{r_{1}-1}(t-s)}\int_{\Omega}u^{p+r_{1}-1}ds-\frac{\mu_{2}q}{16}\int_{s_{0}}^{t}e^{-\frac{q+r_{2}-1}{r_{2}-1}(t-s)}\int_{\Omega}v^{q+r_{2}-1}ds+C_{15}\\
\leq  &-\frac{\mu_{1}p}{16}\int_{s_{0}}^{t}e^{-\frac{p+r_{1}-1}{r_{1}-1}(t-s)}\int_{\Omega}u^{p+r_{1}-1}ds-\frac{\mu_{2}q}{32}\int_{s_{0}}^{t}e^{-\frac{q+r_{2}-1}{r_{2}-1}(t-s)}\int_{\Omega}v^{q+r_{2}-1}ds+C_{16}\\
\leq& C_{16}\ \ \ \mbox{for all}\ t \in(s_{0},T_{max}),
\end{aligned}
\end{equation*}
which together with \eqref{s0} yields \eqref{pq}.
\end{proof}

In view of Lemma \ref{d} and the Moser-type iteration method, the above estimates enable us to infer the following boundedness properties.
\begin{lemma}\label{lemmaintensify}
Under the assumptions of Theorem \ref{theorem1}, there exists a constant $C>0$ independent of $t$ such that the solution to model \eqref{model} satisfies
\begin{equation}\label{intensify1}
\|w(\cdot,t)\|_{W^{1,\infty}(\Omega)}\leq C\ \ \ \mbox{for all}\ t \in(0,T_{max})
\end{equation}
and
\begin{equation}\label{intensify2}
\|u(\cdot,t)\|_{L^{\infty}(\Omega)}+\|v(\cdot,t)\|_{L^{\infty}(\Omega)}\leq C\ \ \ \mbox{for all}\ t \in(0,T_{max}).
\end{equation}
\end{lemma}
\begin{proof}
By \eqref{Lp} and \eqref{pq}, we can find a constant $C_{1}>0$ such that
\begin{equation*}
\|u+v\|_{L^{2N}(\Omega)}\leq C_{1}\ \ \ \mbox{for all}\ t \in(0,T_{max}).
\end{equation*}
Thus Lemma \ref{d} for $\gamma>N$ warrants that \eqref{intensify1} holds. Next, let $p=p_{k}:=2^{k}p_{0}$ with $p_{0}>\frac{N}{2}$ and $M_{k}:=\max\bigg\{1,\sup\limits_{t\in(0,T_{max})}\left(\int_{\Omega}u^{p_{k}}+\int_{\Omega}v^{p_{k}}\right)\bigg\}$. We emphasize that the constants $C_{i}$ $(i\geq2)$ are all independent of $k$. Based on the boundedness of $\| w(\cdot, t)\|_{W^{1,\infty}(\Omega)}$, \eqref{27} entails that
\begin{equation}\label{98}
\begin{aligned}
&\frac{1}{p_{k}}\frac{d}{dt}\int_{\Omega}u^{p_{k}}+(p_{k}-1)\int_{\Omega}u^{p_{k}-2}|\nabla u|^{2}\\
\leq&\chi_{1}(p_{k}-1)C_{2}\int_{\Omega}u^{p_{k}-1}|\nabla u|+C_{2}\int_{\Omega}u^{p_{k}-1}\\
\leq& \chi_{1}(p_{k}-1)C_{2}\int_{\Omega}u^{p_{k}-1}|\nabla u|+C_{2}\int_{\Omega}u^{p_{k}}+C_{3}\\
\leq& \frac{p_{k}-1}{2}\int_{\Omega}u^{p_{k}-2}|\nabla u|^{2}+\left(\frac{\chi_{1}^{2}C_{2}^{2}(p_{k}-1)}{2}+C_{2}\right)\int_{\Omega}u^{p_{k}}+C_{3}\\
\leq& \frac{p_{k}-1}{2}\int_{\Omega}u^{p_{k}-2}|\nabla u|^{2}+C_{4}p_{k}\int_{\Omega}u^{p_{k}}+C_{3}\ \ \ \mbox{for all}\ t \in(0,T_{max}),
\end{aligned}
\end{equation}
where we have applied Young's inequality. Similarly, \eqref{40} follows from Young's inequality that
\begin{equation}\label{99}
\begin{aligned}
&\frac{1}{p_{k}}\frac{d}{dt}\int_{\Omega}v^{p_{k}}+(p_{k}-1)\int_{\Omega}v^{p_{k}-2}|\nabla v|^{2}\\
\leq& \frac{p_{k}-1}{2}\int_{\Omega}v^{p_{k}-2}|\nabla v|^{2}+C_{5}p_{k}\int_{\Omega}v^{p_{k}}+r\int_{\Omega}uv^{p_{k}}+C_{6}\\
\leq& \frac{p_{k}-1}{2}\int_{\Omega}v^{p_{k}-2}|\nabla v|^{2}+C_{5}p_{k}\int_{\Omega}v^{p_{k}}+\frac{rp_{k}}{p_{k}+1}\int_{\Omega}v^{p_{k}+1}+\frac{r}{p_{k}+1}\int_{\Omega}u^{p_{k}+1}+C_{6}
\end{aligned}
\end{equation}
for all  $t\in(0,T_{max})$. Combining \eqref{98} with \eqref{99}, for a constant $C_{7}=2(1-\frac{1}{p_{0}})>0$ a straightforward computation shows
\begin{equation}\label{100}
\begin{aligned}
&\frac{d}{dt}\left(\int_{\Omega}u^{p_{k}}+\int_{\Omega}u^{p_{k}}\right)+\int_{\Omega}u^{p_{k}}+\int_{\Omega}v^{p_{k}}+C_{7}\int_{\Omega}|\nabla u^{\frac{p_{k}}{2}}|^{2}+C_{7}\int_{\Omega}|\nabla v^{\frac{p_{k}}{2}}|^{2}\\
\leq& \left(C_{4}p_{k}^{2}+1\right)\int_{\Omega}u^{p_{k}}+\left(C_{5}p_{k}^{2}+1\right)\int_{\Omega}v^{p_{k}}+\frac{rp_{k}^{2}}{p_{k}+1}\int_{\Omega}v^{p_{k}+1}+\frac{rp_{k}}{p_{k}+1}\int_{\Omega}u^{p_{k}+1}+C_{8}p_{k}\\
\leq& C_{9}p_{k}^{2}\int_{\Omega}u^{p_{k}+1}+C_{10}p_{k}^{2}\int_{\Omega}v^{p_{k}+1}+C_{11}p_{k}^{2}\ \ \ \mbox{for all}\ t \in(0,T_{max}),
\end{aligned}
\end{equation}
because $a^{p_{k}}\leq a^{p_{k}+1}+1$ holds by Young's inequality. Notably, \eqref{100} is similar to (2.12) in \cite{TX2022}. Then we can proceed with the Moser iteration procedure that has the same steps as in \cite{TX2022} to derive \eqref{intensify2}. For brevity, we omit the detailed calculations.
\end{proof}

\textbf{Proof of Theorem 1.1.} Thanks to \eqref{intensify2}, this ensures $T_{max}=\infty$, otherwise contradiction arises from \eqref{local2}. Therefore, we directly establish the global existence and boundedness of classical solution to model \eqref{model}. \ \ \ \ $\blacksquare$

\section{Global existence of the weak solution}
\setcounter{equation}{0}
\indent

This section is concerned with the global existence of weak solution to system \eqref{model} for any $\mu_{i}>0$ and $r_{1}=r_{2}=2$. To better handle the chemotactic terms, we develop the following appropriately regularized problem of \eqref{model}:

\begin{equation}\label{weak model}
\left\{\aligned
  &  u_{\varepsilon t}=\Delta u_{\varepsilon}-\chi_{1}\nabla\cdot(u_{\varepsilon}F_{\varepsilon}(u_{\varepsilon})\nabla w_{\varepsilon})+w_{\varepsilon}-\mu_{1}u_{\varepsilon}^{2},&&x\in\Omega,t>0,\\
  &  v_{\varepsilon t}=\Delta v_{\varepsilon}-\chi_{2}\nabla\cdot(v_{\varepsilon}F_{\varepsilon}(v_{\varepsilon})\nabla w_{\varepsilon})+w_{\varepsilon}+ru_{\varepsilon}v_{\varepsilon}-\mu_{2}v_{\varepsilon}^{2},&&x\in\Omega,t>0,\\
  &  w_{\varepsilon t}=\Delta w_{\varepsilon}+u_{\varepsilon}+v_{\varepsilon}-w_{\varepsilon},&&x\in\Omega,t>0,\\
  &  \frac{\partial u_{\varepsilon}}{\partial \nu}=\frac{\partial v_{\varepsilon}}{\partial \nu}=\frac{\partial w_{\varepsilon}}{\partial \nu}=0,&&x\in\partial\Omega,t>0,\\
  &  u_{\varepsilon}(x,0)=u_{0}(x),v_{\varepsilon}(x,0)=v_{0}(x),w_{\varepsilon}(x,0)=w_{0}(x),&&x\in\Omega,
\endaligned\right.
\end{equation}
where
\begin{equation}\label{F}
F_{\varepsilon}(s)=\frac{1}{(1+\varepsilon s)^{N+1}} \ \ \ \mbox{for all}\ s\geq0 \ \ \mbox{and}\ \ N\geq3.
\end{equation}

We first show the local solvability and extendibility of this system, which can be obtained by using a suitable fixed-point method together with the parabolic regularity theory.
\begin{lemma}\label{Local2}
Let $\Omega\subset\mathbb R^{N}$($N\geq3$) be a bounded domain with smooth boundary, and let the initial data comply with \eqref{initial}. Then one can find a maximal $T_{max,\varepsilon}\in(0,\infty]$ and a uniquely determined triple $(u_{\varepsilon},v_{\varepsilon},w_{\varepsilon})$ satisfying
\begin{equation}\label{local2!}
\left\{\aligned
&u_{\varepsilon}\in C^{0}(\overline{\Omega}\times[0,T_{max,\varepsilon}))\cap C^{2,1}(\overline{\Omega}\times(0,T_{max,\varepsilon})),
&&\\
&v_{\varepsilon}\in C^{0}(\overline{\Omega}\times[0,T_{max,\varepsilon}))\cap C^{2,1}(\overline{\Omega}\times(0,T_{max,\varepsilon})),
&&\\
&w_{\varepsilon}\in C^{0}(\overline{\Omega}\times[0,T_{max,\varepsilon}))\cap C^{2,1}(\overline{\Omega}\times(0,T_{max,\varepsilon})),
\endaligned\right.
\end{equation}
which solves \eqref{weak model} classically in $\Omega\times(0,T_{max,\varepsilon})$. Moreover, if $T_{max,\varepsilon}<\infty$, then
\begin{center}
$\|u_{\varepsilon}(\cdot,t)\|_{L^{\infty}(\Omega)}+\|v_{\varepsilon}(\cdot,t)\|_{L^{\infty}(\Omega)}\rightarrow\infty$ as $t\nearrow T_{max}$.
\end{center}
\end{lemma}

Some useful $\varepsilon$-independent properties of solutions are derived as follows.
\begin{lemma}\label{W1}
Let the conditions in Lemma \ref{Local2} hold. Then there exists a constant $C>0$ independent of $\varepsilon$ such that for any $\varepsilon\in(0,1)$,
\begin{equation}\label{w1}
\int_{\Omega}u_{\varepsilon}(\cdot,t)+\int_{\Omega}v_{\varepsilon}(\cdot,t)+\int_{\Omega}w_{\varepsilon}^{2}(\cdot,t)+\int_{\Omega}|\nabla w_{\varepsilon}(\cdot,t)|^{2}\leq C\ \ \ \mbox{for all}\ t \in(0,T_{max,\varepsilon})
\end{equation}
and
\begin{equation}\label{w2}
\int_{0}^{T}\int_{\Omega}\left(u_{\varepsilon}^{2}+v_{\varepsilon}^{2}+|\nabla w_{\varepsilon}|^{2}+|\Delta w_{\varepsilon}|^{2}\right)\leq C(T)\ \ \ \mbox{for all}\ T \in(0,T_{max,\varepsilon})
\end{equation}
as well as
\begin{equation}\label{w3}
\int_{t}^{t+\tau}\int_{\Omega}\left(u_{\varepsilon}^{2}+v_{\varepsilon}^{2}+|\nabla w_{\varepsilon}|^{2}+|\Delta w_{\varepsilon}|^{2}\right)\leq C\ \ \ \mbox{for all}\ t \in(0,T_{max,\varepsilon}-\tau),
\end{equation}
where $\tau=\min\{1,\frac{1}{4}T_{max,\varepsilon}\}$.

\end{lemma}
\begin{proof}
Applying the same method as in Lemma \ref{lemmaL1} to system \eqref{weak model}, for a positive constant $C_{1}$ and  $L=\frac{1}{2\mu_{2}}r^{2}>0$ we infer that
\begin{equation*}
y(t):=\frac{2L}{\mu_{1}}\int_{\Omega}u_{\varepsilon}(\cdot,t)+\int_{\Omega}v_{\varepsilon}(\cdot,t)+\frac{4L+2\mu_{1}}{\mu_{1}}\int_{\Omega}w_{\varepsilon}(\cdot,t)
\end{equation*}
fulfills
\begin{equation}\label{69}
y'(t)+\frac{1}{2}y(t)+\frac{L}{2}\int_{\Omega}u_{\varepsilon}^{2}+\frac{\mu_{2}}{4}\int_{\Omega}v_{\varepsilon}^{2}\leq C_{1}\ \ \ \mbox{for all}\ t \in(0,T_{max,\varepsilon})\ \mbox{and}\ \varepsilon\in(0,1),
\end{equation}
which in light of an ODE comparison argument implies that
\begin{equation}\label{70}
\int_{\Omega}u_{\varepsilon}+\int_{\Omega}v_{\varepsilon}\leq C_{2}\ \ \ \mbox{for all}\ t \in(0,T_{max,\varepsilon})\ \mbox{and}\ \varepsilon\in(0,1)
\end{equation}
with some constant $C_{2}>0$. An integration of \eqref{69} in time also shows
\begin{equation}\label{71}
\int_{t}^{t+\tau}\int_{\Omega}\left(u_{\varepsilon}^{2}+v_{\varepsilon}^{2}\right)\leq C_{3}\ \ \ \mbox{for all}\ t \in(0,T_{max,\varepsilon}-\tau)\ \mbox{and}\ \varepsilon\in(0,1)
\end{equation}
and
\begin{equation}\label{71!}
\int_{0}^{T}\int_{\Omega}\left(u_{\varepsilon}^{2}+v_{\varepsilon}^{2}\right)\leq C_{4}(T)\ \ \ \mbox{for all}\ T \in(0,T_{max,\varepsilon}),
\end{equation}
where $C_{i}$ $(i=3,4)$ are positive constants and $\tau$ is given in \eqref{w3}. Multiplying the third equation in \eqref{weak model} by $-\Delta w_{\varepsilon}$, we use Young's inequality to obtain
\begin{equation}\label{72}
\begin{aligned}
\frac{1}{2}\frac{d}{dt}\int_{\Omega}|\nabla w_{\varepsilon}|^{2}+\int_{\Omega}|\nabla w_{\varepsilon}|^{2}+\int_{\Omega}|\Delta w_{\varepsilon}|^{2}=& -\int_{\Omega}u_{\varepsilon}\Delta w_{\varepsilon}-\int_{\Omega}v_{\varepsilon}\Delta w_{\varepsilon}\\
\leq&\frac{1}{2}\int_{\Omega}|\Delta w_{\varepsilon}|^{2}+\int_{\Omega}u_{\varepsilon}^{2}+\int_{\Omega}v_{\varepsilon}^{2}.
\end{aligned}
\end{equation}
Thanks to \eqref{71}, an application of Lemma \ref{c} to \eqref{72} yields that
\begin{equation}\label{72!}
\int_{\Omega}|\nabla w_{\varepsilon}|^{2}\leq C_{5}\ \ \ \mbox{for all}\ t \in(0,T_{max,\varepsilon})\ \mbox{and}\ \varepsilon\in(0,1)
\end{equation}
with a constant $C_{5}>0$. Then for some positive constants $C_{6}$ and $C_{7}$ we integrate \eqref{72} in time and apply \eqref{71}-\eqref{71!} to conclude that
\begin{equation}\label{71!!}
\int_{t}^{t+\tau}\int_{\Omega}\left(|\nabla w_{\varepsilon}|^{2}+|\Delta w_{\varepsilon}|^{2}\right)\leq C_{6}\ \ \ \mbox{for all}\ t \in(0,T_{max,\varepsilon}-\tau)\ \mbox{and}\ \varepsilon\in(0,1)
\end{equation}
and
\begin{equation}\label{71!!!}
\int_{0}^{T}\int_{\Omega}\left(|\nabla w_{\varepsilon}|^{2}+|\Delta w_{\varepsilon}|^{2}\right)\leq C_{7}(T)\ \ \ \mbox{for all}\ T \in(0,T_{max,\varepsilon})\ \mbox{and}\ \varepsilon\in(0,1).
\end{equation}
Testing the third equation against $w_{\varepsilon}$, we integrate by parts and use Young's inequality to estimate
\begin{equation}\label{aa}
\begin{aligned}
\frac{1}{2}\frac{d}{dt}\int_{\Omega}w_{\varepsilon}^{2}+\int_{\Omega}|\nabla w_{\varepsilon}|^{2}\leq& \int_{\Omega}u_{\varepsilon}w_{\varepsilon}+\int_{\Omega}v_{\varepsilon}w_{\varepsilon}-\int_{\Omega}w_{\varepsilon}^{2}\\
\leq&\int_{\Omega}u_{\varepsilon}^{2}+\int_{\Omega}v_{\varepsilon}^{2}-\frac{1}{2}\int_{\Omega}w_{\varepsilon}^{2}\ \ \ \mbox{for all}\ t \in(0,T_{max,\varepsilon})\ \mbox{and}\ \varepsilon\in(0,1).
\end{aligned}
\end{equation}
Upon another application of Lemma \ref{c}, \eqref{aa} entails that
\begin{equation}\label{74}
\int_{\Omega} w_{\varepsilon}^{2}\leq C_{8}\ \ \ \mbox{for all}\ t \in(0,T_{max,\varepsilon})\ \mbox{and}\ \varepsilon\in(0,1)
\end{equation}
with a constant $C_{8}>0$, which together with \eqref{70} and \eqref{72!} establishes \eqref{w1}. Finally, \eqref{w2} follows from \eqref{71!} and \eqref{71!!!}, while \eqref{w3} is a consequence of \eqref{71} and \eqref{71!!}.
\end{proof}

We are now in a position to establish the $\varepsilon$-dependent boundedness of $\|u_{\varepsilon}(\cdot,t)\|_{L^{\infty}(\Omega)}$, $\|v_{\varepsilon}(\cdot,t)\|_{L^{\infty}(\Omega)}$ and $\|w_{\varepsilon}(\cdot,t)\|_{W^{1,\infty}(\Omega)}$, which contributes to the well-posedness of system \eqref{weak model}.
\begin{lemma}\label{W3}
Let $\Omega\subset\mathbb R^{N}$ $(N\geq3)$ be a bounded domain with smooth boundary and $\mu_{i}>0$ $(i=1,2)$. For any choice of $\varepsilon\in(0,1)$, the solution to \eqref{weak model} is global in time.
\end{lemma}

\begin{proof}
In this proof, we let $\varepsilon\in(0,1)$ and use $C_{i}$ $(i\in\mathbb{N}^{+})$ to represent some positive constants that may vary at each step and possibly depend on $\varepsilon$. Letting $T_{max,\varepsilon}<+\infty$, it directly follows from \eqref{w3} that
\begin{equation}\label{85}
\int_{0}^{T_{max,\varepsilon}}\int_{\Omega}|\Delta w_{\varepsilon}|^{2}\leq C_{1}.
\end{equation}
Multiplying the first equation in \eqref{weak model} by $u_{\varepsilon}^{p-1}$ $(p>N+1)$ yields that
\begin{equation}\label{86}
\begin{aligned}
&\frac{1}{p}\frac{d}{dt}\|u_{\varepsilon}\|_{L^{p}(\Omega)}^{p}+(p-1)\int_{\Omega}u_{\varepsilon}^{p-2}|\nabla u_{\varepsilon}|^{2}\\
=& (p-1)\chi_{1}\int_{\Omega}u_{\varepsilon}^{p-1}F_{\varepsilon}(u_{\varepsilon})\nabla u_{\varepsilon}\cdot\nabla w_{\varepsilon}+\int_{\Omega}u_{\varepsilon}^{p-1}w_{\varepsilon}-\mu_{1}\int_{\Omega}u_{\varepsilon}^{p+1}\ \ \ \mbox{for all}\ t \in(0,T_{max,\varepsilon}).
\end{aligned}
\end{equation}
By Young's inequality, for all $t\in(0,T_{max,\varepsilon})$ we estimate that
\begin{equation*}
\begin{aligned}
\int_{\Omega}u_{\varepsilon}^{p-1}w_{\varepsilon}\leq&\frac{\mu_{1}}{2}\int_{\Omega}u_{\varepsilon}^{p+1}+C_{2}\int_{\Omega}w_{\varepsilon}^{\frac{p+1}{2}}\\ \leq&\frac{\mu_{1}}{2}\int_{\Omega}u_{\varepsilon}^{p+1}+\frac{1}{8}\int_{\Omega}w_{\varepsilon}^{p}+C_{3}
\end{aligned}
\end{equation*}
and
\begin{equation*}
\begin{aligned}
(p-1)\chi_{1}\int_{\Omega}u_{\varepsilon}^{p-1}F_{\varepsilon}(u_{\varepsilon})\nabla u_{\varepsilon}\cdot\nabla w_{\varepsilon}=&(p-1)\chi_{1}\int_{\Omega}\nabla \left(\int_{0}^{u_{\varepsilon}}\frac{s^{p-1}}{(1+\varepsilon s)^{N+1}}ds\right)\cdot\nabla w_{\varepsilon}\\
\leq& (p-1)\chi_{1}\int_{\Omega}\int_{0}^{u_{\varepsilon}}\frac{s^{p-1}}{(1+\varepsilon s)^{N+1}}ds|\Delta w_{\varepsilon}|\\
\leq& (p-1)\frac{\chi_{1}}{\varepsilon^{N+1}}\int_{\Omega}\int_{0}^{u_{\varepsilon}}s^{p-N-2}ds|\Delta w_{\varepsilon}|\\
\leq& \frac{\chi_{1}(P-1)}{\varepsilon^{N+1}(P-N-1)}\int_{\Omega}u_{\varepsilon}^{p-N-1}|\Delta w_{\varepsilon}|\\
\leq&\frac{\mu_{1}}{4}\int_{\Omega}u_{\varepsilon}^{p+1}+C_{4}\int_{\Omega}|\Delta w_{\varepsilon}|^{\frac{p+1}{N+2}},
\end{aligned}
\end{equation*}
which update \eqref{86} as
\begin{equation}\label{89}
\begin{aligned}
\frac{1}{p}\frac{d}{dt}\|u_{\varepsilon}\|_{L^{p}(\Omega)}^{p}\leq\frac{1}{8}\int_{\Omega}w_{\varepsilon}^{p}+C_{4}\int_{\Omega}|\Delta w_{\varepsilon}|^{\frac{p+1}{N+2}}-\frac{\mu_{1}}{4}\int_{\Omega}u_{\varepsilon}^{p+1}+C_{3}\ \ \ \mbox{for all}\ t \in(0,T_{max,\varepsilon}).
\end{aligned}
\end{equation}

Similarly, multiplying the second equation in \eqref{weak model} by $v_{\varepsilon}^{p-1}$, by means of Young's inequality we obtain
\begin{equation}\label{90}
\begin{aligned}
&\frac{1}{p}\frac{d}{dt}\|v_{\varepsilon}\|_{L^{p}(\Omega)}^{p}+(p-1)\int_{\Omega}v_{\varepsilon}^{p-2}|\nabla v_{\varepsilon}|^{2}\\
=& (p-1)\chi_{2}\int_{\Omega}v_{\varepsilon}^{p-1}F_{\varepsilon}(v_{\varepsilon})\nabla v_{\varepsilon}\cdot\nabla w_{\varepsilon}+\int_{\Omega}v_{\varepsilon}^{p-1}w_{\varepsilon}+r\int_{\Omega}u_{\varepsilon}v_{\varepsilon}^{p}-\mu_{2}\int_{\Omega}v_{\varepsilon}^{p+1}\\
\leq&\frac{1}{8}\int_{\Omega}w_{\varepsilon}^{p}+C_{5}\int_{\Omega}|\Delta w_{\varepsilon}|^{\frac{p+1}{N+2}}+r\int_{\Omega}u_{\varepsilon}v_{\varepsilon}^{p}-\frac{\mu_{2}}{4}\int_{\Omega}v_{\varepsilon}^{p+1}+C_{6}\\
\leq&\frac{1}{8}\int_{\Omega}w_{\varepsilon}^{p}+C_{5}\int_{\Omega}|\Delta w_{\varepsilon}|^{\frac{p+1}{N+2}}+C_{7}\int_{\Omega}u_{\varepsilon}^{p+1}-\frac{\mu_{2}}{8}\int_{\Omega}v_{\varepsilon}^{p+1}+C_{6}\ \ \ \mbox{for all}\ t \in(0,T_{max,\varepsilon})
\end{aligned}
\end{equation}
where some positive constant $C_{7}=\frac{r^{p+1}}{p+1}\left(\frac{\mu_{2}(p+1)}{8p}\right)^{-p}$. Then testing the third equation in \eqref{weak model} against $w_{\varepsilon}^{p-1}$, for any $\varepsilon_{i}>0$ ($i=1,2$) we also apply Young's inequality to see that
\begin{equation}\label{91}
\begin{aligned}
\frac{1}{p}\frac{d}{dt}\|w_{\varepsilon}\|_{L^{p}(\Omega)}^{p}\leq&\int_{\Omega}u_{\varepsilon}w_{\varepsilon}^{p-1}+\int_{\Omega}v_{\varepsilon}w_{\varepsilon}^{p-1}-\int_{\Omega}w_{\varepsilon}^{p}\\
\leq&C_{8}\int_{\Omega}u_{\varepsilon}^{p}+C_{9}\int_{\Omega}v_{\varepsilon}^{p}-\frac{1}{4}\int_{\Omega}w_{\varepsilon}^{p}\\
\leq&\varepsilon_{1}\int_{\Omega}u_{\varepsilon}^{p+1}+\varepsilon_{2}\int_{\Omega}v_{\varepsilon}^{p+1}-\frac{1}{4}\int_{\Omega}w_{\varepsilon}^{p}+C_{10}\ \ \ \mbox{for all}\ t \in(0,T_{max,\varepsilon}).
\end{aligned}
\end{equation}
Define $L_{1}=\frac{4C_{7}}{\mu_{1}}+1$ and $L_{2}=\frac{2C_{7}}{\mu_{1}}+2$. The combination of \eqref{89}, \eqref{90} and \eqref{91} gives
\begin{equation}\label{92}
\begin{aligned}
&\frac{d}{dt}\left(\frac{L_{1}}{p}\|u_{\varepsilon}\|_{L^{p}(\Omega)}^{p}+\frac{1}{p}\|v_{\varepsilon}\|_{L^{p}(\Omega)}^{p}+\frac{L_{2}}{p}\|w_{\varepsilon}\|_{L^{p}(\Omega)}^{p}\right)\\
\leq&C_{11}\int_{\Omega}|\Delta w_{\varepsilon}|^{\frac{p+1}{N+2}}+\frac{1+L_{1}-2L_{2}}{8}\int_{\Omega}w_{\varepsilon}^{p}+\left(\varepsilon_{1}L_{2}+C_{7}-\frac{\mu_{1}L_{1}}{4}\right)\int_{\Omega}u_{\varepsilon}^{p+1}\\
&+\left(\varepsilon_{2}L_{2}-\frac{\mu_{2}}{8}\right)\int_{\Omega}v_{\varepsilon}^{p+1}+C_{12}\ \ \ \mbox{for all}\ t \in(0,T_{max,\varepsilon}).
\end{aligned}
\end{equation}
Let $\varepsilon_{1}=\frac{\mu_{1}}{8L_{2}}$ and $\varepsilon_{2}=\frac{\mu_{2}}{16L_{2}}$.
Based on the choice of $L_{i}$ and $\varepsilon_{i}$ $(i=1,2)$, \eqref{92} implies that
\begin{equation*}
\frac{d}{dt}\left(\frac{L_{1}}{p}\|u_{\varepsilon}\|_{L^{p}(\Omega)}^{p}+\frac{1}{p}\|v_{\varepsilon}\|_{L^{p}(\Omega)}^{p}+\frac{L_{2}}{p}\|w_{\varepsilon}\|_{L^{p}(\Omega)}^{p}\right)\leq C_{11}\int_{\Omega}|\Delta w_{\varepsilon}|^{\frac{p+1}{N+2}}+C_{12}
\end{equation*}
for all $t\in(0,T_{max,\varepsilon})$, which upon an integration yields
\begin{equation}\label{95}
\|u_{\varepsilon}(\cdot,t)\|_{L^{2N}(\Omega)}+\|v_{\varepsilon}(\cdot,t)\|_{L^{2N}(\Omega)}+\|w_{\varepsilon}(\cdot,t)\|_{L^{2N}(\Omega)}\leq C_{13}\ \ \ \mbox{for all}\ t \in(0,T_{max,\varepsilon})
\end{equation}
thanks to \eqref{85} and the choice of $p=2N$. Then we employ Lemma \ref{d} to conclude that
\begin{equation}\label{101}
\|w_{\varepsilon}(\cdot,t)\|_{W^{1,\infty}(\Omega)}\leq C_{14}\ \ \ \mbox{for all}\ t \in(0,T_{max,\varepsilon}).
\end{equation}
Finally, by establishing a Moser-type iteration identical to that in Lemma \ref{lemmaintensify}, one has
\begin{equation}\label{102}
\|u_{\varepsilon}(\cdot,t)\|_{L^{\infty}(\Omega)}+\|v_{\varepsilon}(\cdot,t)\|_{L^{\infty}(\Omega)}\leq C_{15}\ \ \ \mbox{for all}\ t \in(0,T_{max,\varepsilon}).
\end{equation}
Thus Lemma \ref{W3} is a consequence of \eqref{101}, \eqref{102} and the extendibility criterion provided by Lemma \ref{Local2}.
\end{proof}

The following spatio-temporal regularity plays a key role in our further analysis.
\begin{lemma}\label{W2}
Suppose that $\mu_{i}>0$ $(i=1,2)$. Then for any $T>0$, there exists a constant $C(T)>0$ fulfilling
\begin{equation}\label{w21}
\int_{0}^{T}\int_{\Omega}\left(\frac{|\nabla u_{\varepsilon}|^{2}}{u_{\varepsilon}}+\frac{|\nabla v_{\varepsilon}|^{2}}{v_{\varepsilon}}\right)\leq C(T).
\end{equation}
\end{lemma}
\begin{proof}
Testing the first equation in \eqref{weak model} by $\ln u_{\varepsilon}+1$, we integrate by parts to derive
\begin{equation}\label{75}
\begin{aligned}
\frac{d}{dt}\int_{\Omega}u_{\varepsilon}\ln u_{\varepsilon}=&\int_{\Omega}u_{\varepsilon t}\ln u_{\varepsilon}+\int_{\Omega}u_{\varepsilon t}\\
=&-\int_{\Omega}\frac{|\nabla u_{\varepsilon}|^{2}}{u_{\varepsilon}}+\chi_{1}\int_{\Omega}F_{\varepsilon}(u_{\varepsilon})\nabla u_{\varepsilon}\cdot\nabla w_{\varepsilon}+\int_{\Omega}w_{\varepsilon}\ln u_{\varepsilon}+\int_{\Omega}w_{\varepsilon}\\
&-\mu_{1}\int_{\Omega}u_{\varepsilon}^{2}\ln u_{\varepsilon}-\mu_{1}\int_{\Omega}u_{\varepsilon}^{2}\ \ \ \mbox{for all}\ t \in(0,T_{max,\varepsilon})\ \mbox{and}\ \varepsilon\in(0,1),
\end{aligned}
\end{equation}
where due to Young's inequality and \eqref{74} one has
\begin{equation}\label{76}
\begin{aligned}
\int_{\Omega}w_{\varepsilon}\ln u_{\varepsilon}\leq&\varepsilon_{1}\int_{\Omega}(\ln u_{\varepsilon})^{2}+C_{1}\int_{\Omega}w_{\varepsilon}^{2}\\
\leq& \varepsilon_{1}\int_{\Omega}u_{\varepsilon}^{2}+C_{2}
\end{aligned}
\end{equation}
and
\begin{equation}\label{77}
\begin{aligned}
\chi_{1}\int_{\Omega}F_{\varepsilon}(u_{\varepsilon})\nabla u_{\varepsilon}\cdot\nabla w_{\varepsilon}=& \chi_{1}\int_{\Omega}\nabla \left(\int_{0}^{u_{\varepsilon}}\frac{1}{(1+\varepsilon s)^{N+1}}ds\right)\cdot\nabla w_{\varepsilon}\\
\leq& \chi_{1}\int_{\Omega}\left(\int_{0}^{u_{\varepsilon}}\frac{1}{(1+\varepsilon s)^{N+1}}ds\right)|\Delta w_{\varepsilon}|\\
\leq& \chi_{1}\int_{\Omega}u_{\varepsilon}|\Delta w_{\varepsilon}|\\
\leq& \varepsilon_{2}\int_{\Omega}u_{\varepsilon}^{2}+C_{3}\int_{\Omega}|\Delta w_{\varepsilon}|^{2}\\
\end{aligned}
\end{equation}
with some positive constants $C_{i}$ $(i=1, .., 3)$. Inserting \eqref{76}-\eqref{77} into \eqref{75}, for a positive constant $C_{4}=C_{2}+\|w(\cdot,t)\|_{L^{1}(\Omega)}$ we see that
\begin{equation}\label{78}
\begin{aligned}
\frac{d}{dt}\int_{\Omega}u_{\varepsilon}\ln u_{\varepsilon}+\int_{\Omega}\frac{|\nabla u_{\varepsilon}|^{2}}{u_{\varepsilon}}\leq C_{3}\int_{\Omega}|\Delta w_{\varepsilon}|^{2}-\mu_{1}\int_{\Omega}u_{\varepsilon}^{2}\ln u_{\varepsilon}+\left(\varepsilon_{1}+\varepsilon_{2}-\mu_{1}\right)\int_{\Omega}u_{\varepsilon}^{2}+C_{4}
\end{aligned}
\end{equation}
for all $t \in(0,T_{max,\varepsilon})$ and $\varepsilon\in(0,1)$. Similarly, testing the second equation by $\ln v_{\varepsilon}+1$ yields
\begin{equation}\label{79}
\begin{aligned}
&\frac{d}{dt}\int_{\Omega}v_{\varepsilon}\ln v_{\varepsilon}+\int_{\Omega}\frac{|\nabla v_{\varepsilon}|^{2}}{v_{\varepsilon}}\\
=&\chi_{2}\int_{\Omega}F_{\varepsilon}(v_{\varepsilon})\nabla v_{\varepsilon}\cdot\nabla w_{\varepsilon}+\int_{\Omega}w_{\varepsilon}\ln v_{\varepsilon}+\int_{\Omega}w_{\varepsilon}+r\int_{\Omega}u_{\varepsilon}v_{\varepsilon}\\
&+r\int_{\Omega}u_{\varepsilon}v_{\varepsilon}\ln v_{\varepsilon}-\mu_{2}\int_{\Omega}v_{\varepsilon}^{2}\ln v_{\varepsilon}-\mu_{2}\int_{\Omega}v_{\varepsilon}^{2}\\
\leq& r\int_{\Omega}u_{\varepsilon}v_{\varepsilon}+r\int_{\Omega}u_{\varepsilon}v_{\varepsilon}\ln v_{\varepsilon}+C_{5}\int_{\Omega}|\Delta w_{\varepsilon}|^{2}-\mu_{2}\int_{\Omega}v_{\varepsilon}^{2}\ln v_{\varepsilon}+\left(\varepsilon_{3}+\varepsilon_{4}-\mu_{2}\right)\int_{\Omega}v_{\varepsilon}^{2}+C_{6}
\end{aligned}
\end{equation}
with some constants $C_{5}>0$ and $C_{6}>0$, where we have used Young's inequality. For the first two terms on the right side, we can pick $k=\frac{r}{\mu_{2}}+1$ and estimate that
\begin{equation*}
\begin{aligned}
r\int_{\Omega}u_{\varepsilon}v_{\varepsilon}=&r\int_{\{ku_{\varepsilon}\geq v_{\varepsilon}\}}u_{\varepsilon}v_{\varepsilon}+r\int_{\{ku_{\varepsilon}<v_{\varepsilon}\}}u_{\varepsilon}v_{\varepsilon}\\
\leq& rk\int_{\Omega}u_{\varepsilon}^{2}+\frac{r}{k}\int_{\Omega}v_{\varepsilon}^{2}\ \ \ \mbox{for all}\ t \in(0,T_{max,\varepsilon})
\end{aligned}
\end{equation*}
and
\begin{equation*}
\begin{aligned}
r\int_{\Omega}u_{\varepsilon}v_{\varepsilon}\ln v_{\varepsilon}\leq r\int_{\{v_{\varepsilon}>1\}}u_{\varepsilon}v_{\varepsilon}\ln v_{\varepsilon}=&r\int_{\{v_{\varepsilon}>1,ku_{\varepsilon}\geq v_{\varepsilon}\}}u_{\varepsilon}v_{\varepsilon}\ln v_{\varepsilon}+r\int_{\{v_{\varepsilon}>1,ku_{\varepsilon}<v_{\varepsilon}\}}u_{\varepsilon}v_{\varepsilon}\ln v_{\varepsilon}\\
\leq& rk\int_{\Omega}u_{\varepsilon}^{2}\ln ku_{\varepsilon}+\frac{r}{k}\int_{\Omega}v_{\varepsilon}^{2}\ln v_{\varepsilon}
\end{aligned}
\end{equation*}
for all $t\in(0,T_{max,\varepsilon})$ and $\varepsilon\in(0,1)$, which update \eqref{79} as
\begin{equation}\label{82}
\begin{aligned}
&\frac{d}{dt}\int_{\Omega}v_{\varepsilon}\ln v_{\varepsilon}+\int_{\Omega}\frac{|\nabla v_{\varepsilon}|^{2}}{v_{\varepsilon}}\\
\leq& C_{5}\int_{\Omega}|\Delta w_{\varepsilon}|^{2}+rk\int_{\Omega}u_{\varepsilon}^{2}+rk\ln k\int_{\Omega}u_{\varepsilon}^{2}+rk\int_{\Omega}u_{\varepsilon}^{2}\ln u_{\varepsilon}\\
&+\left(\frac{r}{k}-\mu_{2}\right)\int_{\Omega}v_{\varepsilon}^{2}\ln v_{\varepsilon}+\left(\varepsilon_{3}+\varepsilon_{4}+\frac{r}{k}-\mu_{2}\right)\int_{\Omega}v_{\varepsilon}^{2}+C_{6}\ \ \ \mbox{for all}\ t \in(0,T_{max,\varepsilon}).
\end{aligned}
\end{equation}
Let $L=\frac{rk(1+\ln k)}{\mu_{1}}+1$. Combining \eqref{78} with \eqref{82}, some basic computations reveal that
\begin{equation}\label{83}
\begin{aligned}
&\frac{d}{dt}\left(L\int_{\Omega}u_{\varepsilon}\ln u_{\varepsilon}+\int_{\Omega}v_{\varepsilon}\ln v_{\varepsilon}\right)+L\int_{\Omega}\frac{|\nabla u_{\varepsilon}|^{2}}{u_{\varepsilon}}+\int_{\Omega}\frac{|\nabla v_{\varepsilon}|^{2}}{v_{\varepsilon}}\\
\leq& C_{7}\int_{\Omega}|\Delta w_{\varepsilon}|^{2}+(rk-L\mu_{1})\int_{\Omega}u_{\varepsilon}^{2}\ln u_{\varepsilon}+\left(L\varepsilon_{1}+L\varepsilon_{2}+rk+rk\ln k-L\mu_{1}\right)\int_{\Omega}u_{\varepsilon}^{2}\\
&+\left(\frac{r}{k}-\mu_{2}\right)\int_{\Omega}v_{\varepsilon}^{2}\ln v_{\varepsilon}+\left(\varepsilon_{3}+\varepsilon_{4}+\frac{r}{k}-\mu_{2}\right)\int_{\Omega}v_{\varepsilon}^{2}+C_{8}
\end{aligned}
\end{equation}
for all $t\in(0,T_{max,\varepsilon})$ and $\varepsilon\in(0,1)$, where certain positive constants $C_{7}=LC_{3}+C_{5}$ and $C_{8}=LC_{4}+C_{6}$. Then we can fix suitably small $\varepsilon_{i}$ $(i=1, ..., 4)$ fulfilling
\begin{equation*}
0<\varepsilon_{1}+\varepsilon_{2}<\mu_{1}-\frac{rk(1+\ln k)}{L}
\end{equation*}
and
\begin{equation*}
0<\varepsilon_{3}+\varepsilon_{4}<\mu_{2}-\frac{r}{k}.
\end{equation*}
Therefore, for all $t\in(0,T_{max,\varepsilon})$ and $\varepsilon\in(0,1)$, \eqref{83} entails that
\begin{equation}\label{83!}
\frac{d}{dt}\left(L\int_{\Omega}u_{\varepsilon}\ln u_{\varepsilon}+\int_{\Omega}v_{\varepsilon}\ln v_{\varepsilon}\right)+L\int_{\Omega}\frac{|\nabla u_{\varepsilon}|^{2}}{u_{\varepsilon}}+\int_{\Omega}\frac{|\nabla v_{\varepsilon}|^{2}}{v_{\varepsilon}}
\leq C_{7}\int_{\Omega}|\Delta w_{\varepsilon}|^{2}+C_{8}.
\end{equation}
In light of \eqref{71!!!} and the fact $-s\ln s\leq\frac{1}{e}$ for all $s>0$, we readily infer \eqref{w21} by integrating \eqref{83!} over $(0,T)$.
\end{proof}

With the help of the above estimates, we can achieve higher regularity in spatio-temporal estimates of solutions and derive some regularity properties for time derivatives, which facilitate the application of the Aubin-Lions type compactness argument.
\begin{lemma}\label{W4}
For any $T>0$, there exists a constant $C(T)>0$ such that for all $\varepsilon\in(0,1)$,
\begin{equation}\label{w41}
\int_{0}^{T}\int_{\Omega}\left(|\nabla u_{\varepsilon}|^{\frac{4}{3}}+u_{\varepsilon}^{2}\right)+\int_{0}^{T}\int_{\Omega}\left(|\nabla v_{\varepsilon}|^{\frac{4}{3}}+v_{\varepsilon}^{2}\right)\leq C(T)
\end{equation}
and
\begin{equation}\label{w44}
\int_{0}^{T}\int_{\Omega}|u_{\varepsilon}F_{\varepsilon}(u_{\varepsilon})\nabla w_{\varepsilon}|+\int_{0}^{T}\int_{\Omega}|v_{\varepsilon}F_{\varepsilon}(v_{\varepsilon})\nabla w_{\varepsilon}|\leq C(T).
\end{equation}
Furthermore, for any $q>N$ one has
\begin{equation}\label{w42}
\int_{0}^{T}\|u_{\varepsilon t}(\cdot,t)\|_{(W^{2,q}(\Omega))^{\ast}}dt+\int_{0}^{T}\|v_{\varepsilon t}(\cdot,t)\|_{(W^{2,q}(\Omega))^{\ast}}dt\leq C(T)
\end{equation}
and
\begin{equation}\label{w43}
\int_{0}^{T}\|w_{\varepsilon t}(\cdot,t)\|^{2}_{(W^{1,2}(\Omega))^{\ast}}dt\leq C(T).
\end{equation}
\end{lemma}
\begin{proof}
In view of \eqref{w2} and \eqref{w21}, for some positive constants $C_{1}(T)$ and $C_{2}(T)$ we use the H\"{o}lder inequality to derive
\begin{equation*}
\begin{aligned}
&\int_{0}^{T}\int_{\Omega}|\nabla u_{\varepsilon}|^{\frac{4}{3}}+\int_{0}^{T}\int_{\Omega}|\nabla v_{\varepsilon}|^{\frac{4}{3}}\\
=&\int_{0}^{T}\int_{\Omega}\frac{|\nabla u_{\varepsilon}|^{\frac{4}{3}}}{u_{\varepsilon}^{\frac{2}{3}}}u_{\varepsilon}^{\frac{2}{3}}+\int_{0}^{T}\int_{\Omega}\frac{|\nabla v_{\varepsilon}|^{\frac{4}{3}}}{v_{\varepsilon}^{\frac{2}{3}}}v_{\varepsilon}^{\frac{2}{3}}\\
\leq& \left(\int_{0}^{T}\int_{\Omega}\frac{|\nabla u_{\varepsilon}|^{2}}{u_{\varepsilon}}\right)^{\frac{2}{3}}\left(\int_{0}^{T}\int_{\Omega}u_{\varepsilon}^{2}\right)^{\frac{1}{3}}+\left(\int_{0}^{T}\int_{\Omega}\frac{|\nabla v_{\varepsilon}|^{2}}{v_{\varepsilon}}\right)^{\frac{2}{3}}\left(\int_{0}^{T}\int_{\Omega}v_{\varepsilon}^{2}\right)^{\frac{1}{3}}\\
\leq& C_{1}(T)\ \ \ \mbox{for all}\ T >0\ \mbox{and}\ \varepsilon\in(0,1)
\end{aligned}
\end{equation*}
and
\begin{equation*}
\begin{aligned}
&\int_{0}^{T}\int_{\Omega}|u_{\varepsilon}F_{\varepsilon}(u_{\varepsilon})\nabla w_{\varepsilon}|+\int_{0}^{T}\int_{\Omega}|v_{\varepsilon}F_{\varepsilon}(v_{\varepsilon})\nabla w_{\varepsilon}|\\
\leq& \left(\int_{0}^{T}\int_{\Omega}|\nabla w_{\varepsilon}|^{2}\right)^{\frac{1}{2}}\left(\int_{0}^{T}\int_{\Omega}u_{\varepsilon}^{2}\right)^{\frac{1}{2}}+\left(\int_{0}^{T}\int_{\Omega}|\nabla w_{\varepsilon}|^{2}\right)^{\frac{1}{2}}\left(\int_{0}^{T}\int_{\Omega}v_{\varepsilon}^{2}\right)^{\frac{1}{2}}\\
\leq& C_{2}(T)\ \ \ \mbox{for all}\ T >0\ \mbox{and}\ \varepsilon\in(0,1),
\end{aligned}
\end{equation*}
which together with \eqref{w2} yield \eqref{w41} and \eqref{w44}. Testing the first two equations in \eqref{weak model} by certain $\varphi\in C^{1}(\overline{\Omega})$ fulfilling $\|\varphi(\cdot,t)\|_{W^{1,\infty}}\leq1$, by means of \eqref{w1} and Young's inequality we have
\begin{equation}\label{105}
\begin{aligned}
\bigg|\int_{\Omega}u_{\varepsilon t}\varphi\bigg|=&\bigg|-\int_{\Omega}\nabla u_{\varepsilon}\cdot\nabla\varphi+\chi_{1}\int_{\Omega}u_{\varepsilon}F_{\varepsilon}(u_{\varepsilon})\nabla w_{\varepsilon}\cdot\nabla\varphi+\int_{\Omega}(w_{\varepsilon}-\mu_{1}u_{\varepsilon}^{2})\varphi\bigg|\\
\leq&\int_{\Omega}|\nabla u_{\varepsilon}|+\chi_{1}\int_{\Omega}u_{\varepsilon}|\nabla w_{\varepsilon}|+\int_{\Omega}w_{\varepsilon}+\mu_{1}\int_{\Omega}u_{\varepsilon}^{2}\\
\leq&\int_{\Omega}|\nabla u_{\varepsilon}|^{\frac{4}{3}}+(1+\mu_{1})\int_{\Omega}u_{\varepsilon}^{2}+\frac{\chi_{1}^{2}}{4}\int_{\Omega}|\nabla w_{\varepsilon}|^{2}+C_{3}
\end{aligned}
\end{equation}
and
\begin{equation}\label{106!}
\begin{aligned}
\bigg|\int_{\Omega}v_{\varepsilon t}\varphi\bigg|=&\bigg|-\int_{\Omega}\nabla v_{\varepsilon}\cdot\nabla\varphi+\chi_{2}\int_{\Omega}v_{\varepsilon}F_{\varepsilon}(v_{\varepsilon})\nabla w_{\varepsilon}\cdot\nabla\varphi+\int_{\Omega}(ru_{\varepsilon}v_{\varepsilon}+w_{\varepsilon}-\mu_{2}v_{\varepsilon}^{2})\varphi\bigg|\\
\leq&\int_{\Omega}|\nabla v_{\varepsilon}|+\chi_{2}\int_{\Omega}v_{\varepsilon}|\nabla w_{\varepsilon}|+r\int_{\Omega}u_{\varepsilon}v_{\varepsilon}+\int_{\Omega}w_{\varepsilon}+\mu_{2}\int_{\Omega}v_{\varepsilon}^{2}\\
\leq&\int_{\Omega}|\nabla v_{\varepsilon}|^{\frac{4}{3}}+\int_{\Omega}u_{\varepsilon}^{2}+(1+\mu_{2}+\frac{r^{2}}{4})\int_{\Omega}v_{\varepsilon}^{2}+\frac{\chi_{2}^{2}}{4}\int_{\Omega}|\nabla w_{\varepsilon}|^{2}+C_{4}
\end{aligned}
\end{equation}
for all $t >0$ and $\varepsilon\in(0,1)$ with some positive constants $C_{3}$ and $C_{4}$. Due to the Sobolev embedding $W^{2,q}(\Omega)\hookrightarrow W^{1,\infty}(\Omega)$ $(q>N)$ and \eqref{w2}, one can find some positive constants $C_{5}$ and $C_{6}(T)$ such that the combination of \eqref{105} and \eqref{106!} entails
\begin{equation*}
\begin{aligned}
&\int_{0}^{T}\|u_{\varepsilon t}(\cdot,t)\|_{(W^{2,q}(\Omega))^{\ast}}dt+\int_{0}^{T}\|v_{\varepsilon t}(\cdot,t)\|_{(W^{2,q}(\Omega))^{\ast}}dt\\
\leq& C_{5}\left(\int_{0}^{T}\int_{\Omega}|\nabla u_{\varepsilon}|^{\frac{4}{3}}+\int_{0}^{T}\int_{\Omega}|\nabla v_{\varepsilon}|^{\frac{4}{3}}+\int_{0}^{T}\int_{\Omega}u_{\varepsilon}^{2}+\int_{0}^{T}\int_{\Omega}v_{\varepsilon}^{2}+\int_{0}^{T}\int_{\Omega}|\nabla w_{\varepsilon}|^{2}+1\right)\\
\leq&C_{6}(T)\ \ \ \mbox{for all}\ T >0\ \mbox{and}\ \varepsilon\in(0,1),
\end{aligned}
\end{equation*}
which proves \eqref{w42}. Likewise, testing the third equation in \eqref{weak model} against the same $\varphi$, we utilize the Cauchy-Schwarz inequality to see that
\begin{equation}\label{108}
\begin{aligned}
\bigg|\int_{\Omega}w_{\varepsilon t}\varphi\bigg|^{2}=&\bigg|-\int_{\Omega}\nabla w_{\varepsilon}\cdot\nabla\varphi+\int_{\Omega}(u_{\varepsilon}+v_{\varepsilon}-w_{\varepsilon})\varphi\bigg|^{2}\\
\leq&\left\{|\Omega|^{\frac{1}{2}}\cdot\left(\int_{\Omega}|\nabla w_{\varepsilon}|^{2}\right)^{\frac{1}{2}}+|\Omega|^{\frac{1}{2}}\left(\int_{\Omega}u_{\varepsilon}^{2}\right)^{\frac{1}{2}}+|\Omega|^{\frac{1}{2}}\left(\int_{\Omega}v_{\varepsilon}^{2}\right)^{\frac{1}{2}}+|\Omega|^{\frac{1}{2}}\left(\int_{\Omega}w_{\varepsilon}^{2}\right)^{\frac{1}{2}}\right\}^{2}\\
\leq&4|\Omega|\left(\int_{\Omega}|\nabla w_{\varepsilon}|^{2}+\int_{\Omega}u_{\varepsilon}^{2}+\int_{\Omega}v_{\varepsilon}^{2}+\int_{\Omega}w_{\varepsilon}^{2}\right)\ \ \ \mbox{for all}\ t >0\ \mbox{and}\ \varepsilon\in(0,1),
\end{aligned}
\end{equation}
which upon an integration in time yields \eqref{w43} thanks to \eqref{w2}.
\end{proof}

After the above preparations, we are now able to prove the global existence of weak solution to system \eqref{model} by a standard extraction procedure.
\begin{lemma}\label{W5}
For any $\mu_{i}>0$ $(i=1,2)$, there exists a sequence $(\varepsilon_{j})_{j\in\mathbb{N}}\subset(0,1)$ and functions
\begin{equation}\label{weak1}
\left\{\aligned
&u\in L^{2}_{loc}([0,\infty);L^{2}(\Omega))\cap L_{loc}^{\frac{4}{3}}([0,\infty);W^{1,\frac{4}{3}}(\Omega)),
&&\\
&v\in L^{2}_{loc}([0,\infty);L^{2}(\Omega))\cap L_{loc}^{\frac{4}{3}}([0,\infty);W^{1,\frac{4}{3}}(\Omega)),
&&\\
&w\in L^{2}_{loc}([0,\infty);W^{2,2}(\Omega))
\endaligned\right.
\end{equation}
such that $\varepsilon_{j}\searrow0$ as $j\rightarrow\infty$, and that
\begin{equation}\label{w51}
u_{\varepsilon}\rightarrow u\ \ \mbox{a.e.}\ \mbox{in}\ \Omega\times(0,\infty)\ \ \mbox{and}\ \ \mbox{in}\ L_{loc}^{\frac{4}{3}}(\overline{\Omega}\times[0,\infty))
\end{equation}
\begin{equation}\label{w52}
v_{\varepsilon}\rightarrow v\ \ \mbox{a.e.}\ \mbox{in}\ \Omega\times(0,\infty)\ \ \mbox{and}\ \ \mbox{in}\ L_{loc}^{\frac{4}{3}}(\overline{\Omega}\times[0,\infty))
\end{equation}
\begin{equation}\label{w53}
w_{\varepsilon}\rightarrow w\ \ \mbox{a.e.}\ \mbox{in}\ \Omega\times(0,\infty)\ \ \mbox{and}\ \ \mbox{in}\ L_{loc}^{2}(\overline{\Omega}\times[0,\infty))
\end{equation}
\begin{equation}\label{w54}
\nabla w_{\varepsilon}\rightarrow\nabla w\ \ \mbox{a.e.}\ \mbox{in}\ \Omega\times(0,\infty)\ \ \mbox{and}\ \ \mbox{in}\ L_{loc}^{2}(\overline{\Omega}\times[0,\infty))
\end{equation}
\begin{equation}\label{w55}
u_{\varepsilon}\rightharpoonup u\ \mbox{in}\ L_{loc}^{2}(\overline{\Omega}\times[0,\infty))
\end{equation}
\begin{equation}\label{w56}
\nabla u_{\varepsilon}\rightharpoonup \nabla u\ \mbox{in}\ L_{loc}^{\frac{4}{3}}(\overline{\Omega}\times[0,\infty))
\end{equation}
\begin{equation}\label{w57}
v_{\varepsilon}\rightharpoonup v\ \mbox{in}\ L_{loc}^{2}(\overline{\Omega}\times[0,\infty))
\end{equation}
\begin{equation}\label{w58}
\nabla v_{\varepsilon}\rightharpoonup \nabla v\ \mbox{in}\ L_{loc}^{\frac{4}{3}}(\overline{\Omega}\times[0,\infty))
\end{equation}
\begin{equation}\label{w59}
\nabla w_{\varepsilon}\rightharpoonup\nabla w\ \mbox{in}\ L^{2}(\overline{\Omega}\times[0,\infty))
\end{equation}
\begin{equation}\label{w591}
u_{\varepsilon}F_{\varepsilon}(u_{\varepsilon})\nabla w_{\varepsilon}\rightharpoonup u\nabla w\ \mbox{in}\ L^{1}(\overline{\Omega}\times[0,\infty))
\end{equation}
\begin{equation}\label{w592}
v_{\varepsilon}F_{\varepsilon}(v_{\varepsilon})\nabla w_{\varepsilon}\rightharpoonup v\nabla w\ \mbox{in}\ L_{loc}^{1}(\overline{\Omega}\times[0,\infty))
\end{equation}
as $\varepsilon=\varepsilon_{j}\searrow0$ $(j\rightarrow\infty)$, where the triple $(u,v,w)$ is a global weak solution to \eqref{model} in the sense of Definition \ref{weak solution}.
\end{lemma}
\begin{proof}
By Lemmas \ref{W1} and \ref{W4}, for each $T>0$ we can find a positive constant $C_{1}(T)$ such that
\begin{equation}\label{109}
\|u_{\varepsilon}\|_{L^{\frac{4}{3}}((0,T);W^{1,\frac{4}{3}}(\Omega))}\leq C_{1}(T) \ \ \ \mbox{and}\ \ \ \|u_{\varepsilon t}\|_{L^{1}((0,T);(W^{2,q}(\Omega))^{\ast})}\leq C_{1}(T)
\end{equation}
and
\begin{equation}\label{110}
\|v_{\varepsilon}\|_{L^{\frac{4}{3}}((0,T);W^{1,\frac{4}{3}}(\Omega))}\leq C_{1}(T) \ \ \ \mbox{and}\ \ \ \|v_{\varepsilon t}\|_{L^{1}((0,T);(W^{2,q}(\Omega))^{\ast})}\leq C_{1}(T)
\end{equation}
as well as
\begin{equation}\label{111}
\|w_{\varepsilon}\|_{L^{2}((0,T);W^{2,2}(\Omega))}\leq C_{1}(T) \ \ \ \mbox{and}\ \ \ \|w_{\varepsilon t}\|_{L^{2}((0,T);(W^{1,2}(\Omega))^{\ast})}\leq C_{1}(T).
\end{equation}
Then for any $T>0$, the application of the Aubin-Lions type lemma (see \cite{S1986}) can ensure the strong precompactness of $(u_{\varepsilon})_{\varepsilon\in(0,1)}$ and $(v_{\varepsilon})_{\varepsilon\in(0,1)}$ in $L^{\frac{4}{3}}(\overline{\Omega}\times(0,T))$, and of $(w_{\varepsilon})_{\varepsilon\in(0,1)}$ in $L^{2}((0,T);W^{1,2}(\Omega))$. These strong compactness properties enable us to pick a sequence $(\varepsilon_{j})_{j\in\mathbb{N}}\subset(0,1)$ such that \eqref{w51}-\eqref{w54} hold as $\varepsilon=\varepsilon_{j}\searrow0$ $(j\rightarrow\infty)$. The boundedness results \eqref{w2} and \eqref{w41} also provide a subsequence fulfilling \eqref{w55}-\eqref{w59}.

Now \eqref{w54} together with \eqref{w51}-\eqref{w52} and \eqref{F} provides a subsequence still labeled as $(\varepsilon_{j})_{j\in\mathbb{N}}\subset(0,1)$ such that
\begin{equation}\label{112}
u_{\varepsilon}F_{\varepsilon}(u_{\varepsilon})\nabla w_{\varepsilon}\rightarrow u\nabla w\ \ \mbox{a.e.}\ \mbox{in}\ \Omega\times(0,\infty)\ \ \mbox{as}\ \varepsilon=\varepsilon_{j}\searrow0
\end{equation}
and
\begin{equation}\label{113}
v_{\varepsilon}F_{\varepsilon}(v_{\varepsilon})\nabla w_{\varepsilon}\rightarrow v\nabla w\ \ \mbox{a.e.}\ \mbox{in}\ \Omega\times(0,\infty)\ \ \mbox{as}\ \varepsilon=\varepsilon_{j}\searrow0.
\end{equation}
According to \eqref{w44}, one also deduces that
\begin{equation*}
u_{\varepsilon}F_{\varepsilon}(u_{\varepsilon})\nabla w_{\varepsilon}\rightharpoonup z_{1}\ \ \mbox{in}\ L_{loc}^{1}(\overline{\Omega}\times[0,\infty))\ \ \mbox{as}\ \varepsilon=\varepsilon_{j}\searrow0
\end{equation*}
and
\begin{equation*}
v_{\varepsilon}F_{\varepsilon}(v_{\varepsilon})\nabla w_{\varepsilon}\rightharpoonup z_{2}\ \ \mbox{in}\ L_{loc}^{1}(\overline{\Omega}\times[0,\infty))\ \ \mbox{as}\ \varepsilon=\varepsilon_{j}\searrow0.
\end{equation*}
By \eqref{112}-\eqref{113}, we infer from the Egorov theorem that $z_{1}=u\nabla w$ and $z_{2}=v\nabla w$, which prove \eqref{w591}-\eqref{w592}.

Therefore, \eqref{w56}, \eqref{w58} and \eqref{w591}-\eqref{w592} can ensure the integrability of $\nabla u$, $\nabla v$, $u\nabla w$ and $v\nabla w$, and \eqref{w51}-\eqref{w59} warrant the regularity requirements in Definition \ref{weak solution}. Based on the above convergence properties, from a limit procedure we readily show the integral properties \eqref{i1}-\eqref{i3}. Now, we are in a position to establish a weak solution to \eqref{model} in the claimed sense.
\end{proof}

\textbf{Proof of Theorem 1.2.}  Theorem 1.2 follows trivially from Lemma \ref{W5}.\ \ \ \ $\blacksquare$

\section*{Acknowledgment}
The authors would like to thank the anonymous referees whose comments help to improve the contain of this paper. This work was supported by Shandong Provincial Natural Science Foundation (No. ZR2022JQ06) and the National Natural Science Foundation of China (11601215).

\section*{Conflict of interest statement}
The authors declare no conflicts of interest.

\end{document}